\documentclass[12pt]{article}
\usepackage[alphabetic]{amsrefs}
\usepackage{fancyhdr}
\usepackage{mathema}
\usepackage{upgreek}
\usepackage{soul}
\usepackage{enumitem}
\usetikzlibrary{positioning}
\usetikzlibrary{shapes.misc}
\usepackage{tkz-graph}
\usepackage[many]{tcolorbox}
\usepackage{wrapfig}
\usepackage{lipsum}

\newenvironment{acks}{%
  \begin{abstract}
}{%
  \end{abstract}
}

\newtcolorbox{story}[1][]{
  width=\textwidth,
  fonttitle=\bfseries,
  breakable,
  fonttitle=\bfseries\color{Brown},
  colframe=Melon,
  colback=Melon!10
  #1}

\newtcolorbox[use counter=mynote]
  {mynote}[1][]
  {title=Note~\thetcbcounter,
   width=4cm,
   left=0pt,
   right=0pt,
   fonttitle=\bfseries,
   coltitle=black,
   colframe=ForestGreen!40,
   colback=ForestGreen!10,
   #1
}

\begin{document}
\title{Non-Affine Stein Manifolds and Normal Crossing Divisors}
\author{Randall R. Van Why}
\date{\today}

\maketitle

\begin{abstract}
We show that there are Stein manifolds that admit normal crossing divisor compactifications despite being neither affine nor quasi-projective. To achieve this, we study the contact boundaries of neighborhoods of symplectic normal crossing divisors via a contact-geometric analog of W. Neumann’s plumbing calculus. In particular, we give conditions under which the neighborhood is determined by the contact structure on its boundary.
\end{abstract}
\tableofcontents
\begin{acks}
This work was primarily supported by NSF Grant DMS-1502632. We thank the institutions of Northwestern University, Princeton University, and the Georgia Institute of Technology for the hospitality during the completion of this project.  Much thanks to Austin Christian, John Etnyre, David Gay, Jie Min, Emmy Murphy, Hadrian Quan, and Eric Zaslow for the numerous helpful conversations and input.
\end{acks}

\section{Motivation and summary of results}

We seek to study compactifications of a Stein surface $X$. By compactification, we mean a complex analytic 4-manifold $M$ admitting a complex analytic submanifold $D \subset M$ such that $M - D$ is biholomorphic to $X$. The consideration of these objects gives rise to some interesting motivating questions:
\begin{itemize}
    \item When is a Stein manifold $X$ an algebraic variety (affine or quasi-projective)?
    \item When does a Stein manifold admit an algebraic compactification?
    \item Which algebraic Stein manifolds admit non-algebraic compactifications?
    \item How are the compactifications of $X$ related to each other?
\end{itemize}
In this article, we consider these questions purely in the context of symplectic geometry where the rigidity of symplectic structures can lend a hand. Every Stein manifold $X$ admits a symplectic structure and the general symplectic-topological underpinnings of Stein manifolds were characterized by Eliashberg in \cite{Eliashberg:stein}.

\subsection{Compactifications and non-K\"{a}hler manifolds}

If $X$ is a complex affine variety, we may always \emph{projectivize} $X$ to obtain a (possibly singular) projective variety $\widetilde{M}$, its projective completion. By construction, the intersection of the projective completion with an affine chart in projective space is biholomorphic to our original affine variety $X$. The set $\widetilde{D} = \widetilde{M} - X$ is a subvariety of $\widetilde{M}$ called a \emph{compactifying divisor}. 

It is a classical result of Hironaka \cite{Hironaka:resolution2, Hironaka:resolution1} that the singularities of $\widetilde{M}$ can be completely resolved so as to produce a smooth projective variety $M$. This resolution process is carried out by blowing up points lying on the compactifying divisor and can be performed in such a way that the resulting manifold $M$ still contains a biholomorphic copy of $X$. The complement $D := M - X$ is a subvariety with particularly nice singularities: (simple) normal-crossing singularities. The main topological conclusion we can draw from this story is summarized by:

    \textbf{Fact:} Every affine variety $X$ admits a smooth (projective) compactification $M$ by a normal crossing divisor.

A similar picture emerges with quasi-projective varieties which may be presented as the complements of normal crossing divisors in projetive varieties. Today, a symplectic geometer would recognize a Stein manifold as an example of a Liouville manifold and, through that perspective, may exploit any number of different symplectic invariants. In \cite{McLean:growth}, via an approach first introduced by Seidel \cite{Seidel:symplectic}, McLean established an obstruction to a Stein manifold $X$ being symplectomorphic to an affine variety relying on the growth rate of symplectic homology (c.f. \cite{Seidel:symplectic}). Through use of this obstruction, McLean has been able to determine, for example, when a cotangent bundle may be symplectomorphic to an affine variety \cite[Corollary~1.3]{McLean:growth}. These growth rate techniques take a view toward the above compactification phenomenon. 

McLean and Seidel's techniques essentially determine whether this sort of compactification is even possible for a given Liouville manifold $X$. In particular, they obstruct $X$ from being symplectomorphic to an affine variety by showing it cannot be compactified by a normal crossing divisor. One would naively hope that the existence of such a compactification is sufficient for $X$ to be symplectomorphic to some affine (or at least quasi-projective) variety. Our results will demonstrate that this is not the case.

 \begin{proof}[\normalfont \textbf{Theorem ~\ref{thm:main_example}}]
\renewcommand{\qedsymbol}{}
\em
There exists a Stein manifold admitting a normal crossing divisor compactification that is not biholomorphic to any affine variety.
\end{proof}

To describe this Stein manifold more thoroughly, we recall a classic construction construction in sympelectic geometry. In \cite{thurston:example}, Thurston presented the first example of a closed non-K\"{a}hler symplectic manifold. Before the explication of this example, it was not known whether there were any significant differences between K\"{a}hler and symplectic manifolds. Thurston's example is 4-dimensional and was first studied in a separate context by Kodaira \cite{Kodaira:structure} who showed that it admits an integrable almost complex structure and is thus a complex manifold. Because of this, the example has been known in the literature as the \emph{Kodaira-Thurston manifold} which we will denote by $M_{KT}$. 

In Kodaira's original presentation, we have $M_{KT} := (Nil^3/H) \times S^1$ where $Nil^3$ is the Heisenberg group and $H \subset Nil^3$ is a lattice. We will adopt Thurston's presentation in \cite{thurston:example} as the product $M_{KT} := \MMM_\phi^3 \times S^1$ where 
$\MMM_{\phi}^3$ is the \emph{mapping torus} of a right handed Dehn twist $\phi \colon T^2 \to T^2$. The diffeomorphism type of $\MMM_{\phi}$ is independent of the curve about which we perform the Dehn twist. The manifold $M_{KT}$ cannot be K\"{a}hler because we have $b_1(M_{KT}) = 3$ \cite{thurston:example}. K\"{a}hler manifolds possess a natural Hodge structure and their first Betti number is always even \cite[Chapter~7]{GH:principles}.

In \cite{Gompf:construction}, Gompf showed that, for even dimensions $2n \ge 4$, any finitely presented group can be made the fundamental group of some closed non-K\"{a}hler symplectic $2n$-manifold. Gompf's construction presents an avenue for finding many non-K\"{a}hler examples using obstructions similar to Thurston's. With all of this in mind, it is clear that being K\"{a}hler imposes restrictions on the underlying symplectic geometry, at least on the level of fundamental groups and first homology.

\subsection{An interesting non-algebraic Stein manifold}
Using the results contained in this work, we will construct a non-algebraic Stein 4-manifold $X_{KT}$ given as the complement of a normal crossing divisor in the Kodaira-Thurston example. This will prove our original theorem.
\begin{thm}\label{thm:main_example}
  Let $M_{KT}$ be the Kodaira-Thurston manifold (see \cite{thurston:example}). Let $D_{KT} \subset M_{KT}$ denote the normal crossing divisor given by the union of a generic section with a fiber of $M_{KT}$, thought of as an elliptic fibration. Then the Stein manifold  $X_{KT} = M_{KT} - D_{KT}$ is not biholomorphic to any affine variety.
\end{thm}
If $X_{KT}$ were indeed biholomorphic to an affine variety $X$, then $X$ would admit a compact subdomain $\overline{X}\subset X$ whose interior is symplectomorphic to $X_{KT}$. We are able to eliminate the possibility of the existence of such an $X$ by studying the space of all \emph{symplectic} normal crossing (or $SNC^+$) divisor compactifications of the Liouville domain $\overline{X}$ and relating them. We do this by considering gluings $M = \overline{X} \cup_{\partial \overline{X}} N_D$ where $(N_D,\omega)$ is a regular symplectic neighborhood of a normal crossing divisor $D$. The Stein structure of $X$ endows $\overline{X}$ with the structure of a Liouville domain, a compact exact symplectic manifold with boundary. This structure furnishes $\partial \overline{X}$ with a natural contact structure. Because of this, we can phrase this as the study of $SNC^+$ divisor compactifications of the Liouville domain $\overline{X}$ obtained via a gluing along its contact boundary. We will adopt this context for the rest of our discussion.

\subsection{Symplectic and contact topological results}
Each $SNC^+$ divisor $D$ has an associated decorated graph $\Gamma_D$ which encodes topological information about the divisor. We call $\Gamma_D$ the \emph{divisor graph} of $D$. The main question we are interested in is:

\textbf{Question:} How much about the isomorphism type of $\Gamma_D$ can we learn from the contactomorphism type of $(\partial \overline{X}, \xi)$?

We can see easily that answering this question makes steps toward understanding all $SNC^+$ compactifications of $\overline{X}$. Any such compactification $(M,D)$ has an implicit contactomorphism $\Psi \colon (\partial \overline{X},\xi_D) \to (\partial N_D, \xi_D)$ which defines the gluing between $\overline{X}$ and $N_D$. If we were able to discern information about $\Gamma_D$ purely from this gluing map, we may be able to restrict the class of $SNC^+$ divisors $D$ that may compactify $X$. In other words, we would understand something about the structure of the collection of all compactifications of $X$.

If one neglects symplectic, contact, and algebraic structures entirely and focuses only on the topological aspects of the above discussion, this perspective is well understood. The 3-manifolds $Y_D := \partial N_D$, henceforth known as \emph{divisor boundaries}, are examples of \emph{3-manifold plumbings} which are classical objects in low dimensional topology. The graphs 
$\Gamma_D$ are examples of \emph{plumbing graphs} which give a combinatorial description of how to perform a 3-manifold plumbing and construct $Y_D$ topologically. In our discussion, we shall use ``divisor graphs" and ``plumbing graphs" essentially interchangeably with a preference toward ``divisor graph." We will distinguish the two wherever necessary.

In \cite{Neumann:calculus}, it was shown that divisor graphs can be ``reduced" to simpler divisor graphs via a well-defind procedure without changing the diffeomorphism type of its associated divisor boundary. This reduction procedure always terminates and we are left with a graph $\Gamma_D^{Top}$ which we will call the \emph{topological reduction} of $\Gamma_D$. Neumann proved that, for most 3-manifold plumbings, the diffeomorphism type of the plumbing characterizes the isomorphism type of $\Gamma_D^{Top}$. Rephrased in terms of divisor boundaries, this theorem reads: 
\begin{thm}[c.f. \cite{Neumann:calculus}, Theorem 4.2]
    Let $Y_D$ and $Y_{\widetilde{D}}$ divisor boundaries with divisor graphs $\Gamma_{D}$ and $\Gamma_{\widetilde{D}}$. Then $Y_D$ and $Y_{\widetilde{D}}$ are diffeomorphic if and only if $\Gamma_D^{Top} \approx \Gamma_{\widetilde{D}}^{Top}$.
\end{thm}
Throughout, we shall refer to this result as ``Neumann's theorem." While this result does answer the question topologically, it is less useful for our purposes since Neumann's reduction procedure has no regard (and actually may destroy) the contact topology at hand. To deal with this, we develop a new reduction procedure which terminates and produces a (possibly distinct) graph $\Gamma_D^{\xi}$, the \emph{contact reduction} of $\Gamma_D$. In the case when the divisor boundary associated to $\Gamma_D$ is a \emph{prime} 3-manifold (c.f. \cite{Hatcher:3manifolds}, Chapter 1), by comparing $\Gamma_D^\xi$ and $\Gamma_D^{Top}$, we can arrive at some topological conclusions:
\begin{thm}
Let $N_{D}$ be a concave divisor neighborhood (\Cref{sec:concavity}). Then if $\partial N_D$ is a prime 3-manifold and $\Gamma^{\xi}_D \not \approx \Gamma^{Top}_D$, one or more of the following are true:
\begin{enumerate}[label=(\roman*)]\label{thm:divisor_boundary}
    \item We may blow up $N_D$ to a concave $SNC^+$ divisor neighborhood $N_{\widetilde{D}}$ which contains a smooth component $S$ diffeomorphic to a sphere with $S\cdot S = 0$,
    \item $(\partial N_D, \xi_{D})$ is Seifert-fibered over $\RP^2$ via a fibration with at most one singular fiber, or
    \item $(\partial N_D, \xi_{D})$ contains an embedded incompressible Klein bottle.
\end{enumerate}
\end{thm}
This structural theorem is implied by our main results for divisor boundaries (which are implicit in this result if one restricts to statements only about $\partial N_D$). The details of this will be discussed in \Cref{sec:birational}.

\begin{defn}
    We say a divisor boundary $(Y_{D}, \xi_{D})$ is \emph{obstructed} if $\Gamma^{\xi} \not \approx \Gamma^{Top}$ and  \emph{unobstructed} otherwise. 
\end{defn}

In the unobstructed case, the following result follows directly from \cite{Neumann:calculus}.
\begin{thm}\label{thm:unique_reduction}
Let $(Y_{D}, \xi_{D})$ and $(Y_{\widetilde{D}}, \xi_{\widetilde{D}})$ be a pair of unobstructed divisor boundaries with divisor graphs $\Gamma_D$ and $\Gamma_{\widetilde{D}}$. Suppose that $(Y_D,\xi_D)$ and $(Y_{\widetilde{D}},\xi_{\widetilde{D}})$ are contactomorphic. Then $\Gamma_D^{\xi} \approx \Gamma_{\widetilde{D}}^{\xi}$.
\end{thm}
Thus \Cref{thm:divisor_boundary} and \Cref{thm:unique_reduction} give a complete topological characterization of the extent to which Neumann's theorem fails to hold in the contact category. The proof of \Cref{thm:main_example} essentially follows from a corollary of \Cref{thm:unique_reduction}. The unobstructed case allows for to draw a strong conclusion about divisor neighborhoods:
\begin{thm}\label{thm:unique_nbhd}
Let $(N_D,\omega)$ and $(N_{\widetilde{D}},\widetilde{\omega})$ denote concave divisor neighborhoods (\Cref{sec:concavity}) and suppose that 
$(\partial N_D, \xi_D)$ and $(\partial N_{\widetilde{D}}, \xi_{\widetilde{D}})$ are their associated divisor boundaries. Suppose that both $D$ and $\widetilde{D}$ are unobstructed. Then if $(\partial N_D,\xi_{D})$ is contactomorphic to $(\partial N_{\widetilde{D}},\xi_{\widetilde{D}})$, the symplectic manifolds $(N_D,\omega)$ 
and $(N_{\widetilde{D}},\widetilde{\omega})$ are diffeomorphic up to blow-ups and blow-downs.
\end{thm}
The majority of the work in verifying that $X_{KT}$ is non-algebraic, that is proving \Cref{thm:main_example}, is reduced to verifying that the divisor boundary associated to $(M_{KT}, D_{KT})$ is unobstructed in the above sense. From there we can use \Cref{thm:unique_nbhd} to understand all concave $SNC^+$ divisor compactifications of $X_{KT}$ which will allow us to conclude \Cref{thm:main_example}. The details of this will be handled in the latter part of \Cref{sec:birational}.

To draw conclusions about compactifications, there is a subtlety involving mapping classes of diffeomorphisms used to glue on a divisor neighborhood. This is managed in the following result.
\begin{thm}\label{thm:main}
    Let $D$ be a concave compactifying $SNC^+$ divisor for a 4-dimensional Liouville domain $(X^4,\omega)$ and let $(M^4,\omega)$ be a compactification of $X$ by $D$ with associated mapping class $[\Psi] \in \pi_0(Cont(\partial X))$ defining the capping. If $D$ is unobstructed, any other $SNC^+$ divisor compactification $(\widetilde{M}, \widetilde{\omega})$ with the same mapping class $[\Psi]$ is diffeomorphic to $(M,\omega)$ up to blow-ups and blow-downs.
\end{thm} 

The outline of the paper is as follows:
We will start with some necessary preliminaries (\Cref{sec:prelim}) surrounding the symplectic and contact geometric theory of concave symplectic normal crossing (or $SNC^+$) divisors. We then turn to the classical theory of plumbings, which will make up the main theoretical context for proving our primary lemma on divisor graph reductions. We will outline a framework and process 
for normalizing divisor graphs while preserving the contact structure on the divisor boundary (\Cref{sec:cpc}). After showing that we can almost normalize any divisor graph, we will collect all the topological ramifications into the main results of the paper (\Cref{sec:birational}). We will then focus our attention to the particular case $(M_{KT}, D_{KT})$ above. We will use \Cref{thm:unique_nbhd} to draw conclusions about $X_{KT}$ and $D_{KT}$ that will allow us to conclude with \Cref{thm:main_example}. A short discussion on birational symplectic geometry will follow.

\section{Preliminaries}\label{sec:prelim}
This section collects short definitions and descriptions of the geometric objects relevant to the sections that follow. Everything will be phrased in low dimensions as we will only focus on 4-dimensional varieties and symplectic manifolds with 3-dimensional contact-type boundaries. As a result, the divisors we concern ourselves with are all symplectic surface configurations.

\subsection{Positive symplectic normal crossing divisors}
Let $(M,\omega)$ be a symplectic 4-manifold. An embedded closed 
 symplectic surface $\Sigma \subset M$ is called 
a \emph{symplectic divisor}. 

\begin{defn}
  A subspace $D \subset (M^4,\omega)$ is called a (simple) \emph{symplectic divisor with positive normal crossings} 
  (or simply an $SNC^+$-divisor) if we may write 
  \[D = \bigcup_{i=1}^N D_i\]
  where $\{D_i\}_{i=1}^N$ is a collection of embedded closed symplectic surfaces in $M$ such that any pair $(D_i,D_j)$ with  
  $D_i \cap D_j \neq \emptyset$ intersects $\omega$-orthogonally only in transverse double points (i.e. all intersections must be modeled after intersecting coordinate planes in $\CC^2$). The manifolds $D_i$ are called the \emph{smooth components} of the divisor
  $D$.
\end{defn}

To each $SNC^+$ divisor $D$, we may construct a finite
graph $\Gamma_D$ encoding the connectivity information of $D$ as follows: 
\begin{itemize}
  \item The vertex set of $\Gamma_D$ consists of one vertex $v_i$ for 
  every component $D_i$.
  \item Two distinct vertices $v_i,v_j$ are connected by an edge $e_{ij}^{k}$ for each $p_k \in D_i \cap D_j \neq \emptyset$.
\end{itemize}
It is important to note that, with this definition, there are no edges in $\Gamma_D$ from a vertex to itself (i.e. $\Gamma_D$ is free of loops). Thus $SNC^+$ divisors, as defined above, are topologically identical to simple normal crossing divisors in the complex-algebro-geometric setting.

\begin{rem}
  For notational convenience, we will re-label the components of $D$ by the corresponding vertices of $\Gamma_D$. Any vertex $v \in \Gamma_D$ 
    corresponds to a component $D_{v} \subset D$. We let $\mathcal{E}(v)$ denote the set of edges in $\Gamma_D$ adjacent to $v$ and we let 
    $\mathcal{N}(v)$ be the set of vertices neighboring $v$.
\end{rem}

Regular neighborhoods of symplectic divisors are modeled after plumbings of disk bundles over the components. We let $N_D$ be such smooth plumbing. The boundary $\partial N_D$ is called a \emph{divisor boundary} and is denoted by $Y_D := \partial N_D$ where $N_D$ is the neighborhood of $D$ discussed above. 

\subsection{Symplectic concavity near the boundary of a divisor}\label{sec:concavity}
In many cases, closed neighborhoods of $SNC^+$ divisors will admit natural contact structures on their boundaries. These contact structures are due to the symplectic geometry near the boundary of the neighborhood. We say that the symplectic manifold with boundary $N_D$ is \emph{concave} near $\partial N_D$ if it admits a collar neighborhood symplectomorphic to a portion of the negative symplectization of a contact manifold. In particular, this implies the presenece of a local Liouville vector field $V$ which is inwardly transverse to $\partial N_D$ and has the proprty that $L_V\omega = -\omega$. This means the flow contracts the symplectic area over time.

\begin{defn}
  An $SNC^+$ divisor $D$ is said to be \emph{concave} if every regular 
  neighborhood $N_D$ admits a closed sub-neighborhood $N_D$ such that $(N_D, \omega|_{N_D})$ is concave near $\partial N_D$.
\end{defn}

The Liouville vector field induces a contact structure $\xi_D$ on the boundary $Y_D := \partial N_D$. The contactomorphism type of $(Y_D,\xi_D)$ is stable under perturbation of $V$. 

There is a simple condition that determines when $N_D$ is concave. Every $SNC^+$ divisor, concave or otherwise, has an associated \emph{intersection form} $Q_{D} \colon \Gamma_D \times \Gamma_D \to \ZZ$ 
defined by
\[Q_D(v,w) = D_v \cdot D_w\]
From this definition, $Q_D(v,w) = d$ if $\Gamma_D$ has $d \ge 0$ edges from $v$ to $w$ and $Q_D(v,v)$ is the self-intersection number (in $M$) of $D_v$. 
Placing an arbitrary order on the vertices $v_1,\ldots, v_N$, we regard the intersection form as a linear map $Q_D \colon \RR^N \to \RR^N$ 
defined, as a matrix, by the entries 
\[(Q_D)_{ij} = Q_{D}(v_i,v_j).\]
We also have a well defined \emph{area vector} $\underline{a}\in \RR^N$ with entries 
\[a_{v_i} = \int_{D_{v_i}} \omega.\]
\begin{defn}
  We say that $D$ satisfies \emph{the positive GS-criterion} if there exists a vector $\underline{b} \in \RR^N$ with positive entries solving
    \[Q_D\underline{b} = \underline{a}.\]
\end{defn}

We then have:

\begin{thm}[\cite{LM:capping}, \textbf{Theorem 1.1}]
  If $D$ satisfies the positive GS-criterion, then 
  $D$ has a concave neighborhood $N_D$ inside any regular neighborhood $U$ via the GS-construction (\cite{LM:local}, Section 2.1). 
  The contactomorphism type of the boundary of $N_D$ is independent of all choices made during the GS-construction.
\end{thm}
\begin{rem}``GS" refers to the authors of \cite{GS:surgeries} wherein a separate condition involving negative solutions is shown to be sufficient to ensure the existence of a \emph{convex} neighborhood. 
\end{rem}

Keeping track of the full symplectic topological data of concave neighborhoods requires decorations on the graph $\Gamma_D$ in the form of:
\begin{itemize}
    \item Functions $(g,k) \colon Vert(\Gamma_D) \to \NN_{\ge 0} \times \ZZ$ where 
    $g(v)$ is the genus and $k(v)$ is the self intersection number of the divisor $D_v$ and $Vert(\Gamma_D)$ is the set of vertices of $\Gamma_D$.
    \item The symplectic area vector $\underline{a}$.
\end{itemize}
For the remainder of our discussion, all divisors will assumed to be concave and so we will assume that their (decorated) graphs satisfy the positive GS-criterion.

\subsection{Divisor Boundaries }\label{sec:tpc}

The fact that closed regular neighborhoods of concave divisors admit contact type boundaries means that to every such divisor $D$, we may associate a contact 3-manifold $(Y_D, \xi_D)$. The contactomorphism type of $(Y_D,\xi_D)$ is stable up to perturbation of the local Liouville vector field. Any such 3-manifold $Y_D$ is called a \emph{divisor boundary} and the contact manifold $(Y_D,\xi_D)$ is called a \emph{contact divisor boundary}. As previously mentioned, divisor boundaries are examples of 3-manifold 
plumbings. These objects are classical to topology and have received extensive attention and use.

The diffeomorphism type of $Y_D$ is completely determined by the divisor $D$. This follows from classical results in smooth topology. We are interested in understanding to what extent the diffeomorphism type of $Y_D$ determines the contactomorphism type of $(Y_D,\xi_D)$.

\subsection{Topological Plumbing Calculus}
The diffeomorphism type of a divisor boundary (or even a more general plumbing) does not determine the isomorphism type of the divisor graph $\Gamma_D$. Performing a symplectic blow-up on $(N_D,d\lambda)$ at some point $p \in D$ away from the singular points will produce a new symplectic 4-manifold $(N_{\widetilde{D}}, \widetilde{\omega})$ which is a regular neighborhood of a divisor $\widetilde{D}$ which reduces to $D$ after blowing-down the exceptional curve. Blowing up is a local construction and does not affect the diffeomorphism type of the boundary $\partial N_{\widetilde{D}}$. We thus have:
\begin{prop}\label{prop:non_unique}
    If $Y_D$ and $Y_{\widetilde{D}}$ are divisor boundaries, then we may have $Y_D \approx_{\text{Diff}} Y_{\widetilde{D}}$ with $\Gamma_D \not \approx \Gamma_{\widetilde{D}}$.
\end{prop}
We conclude that we cannot possibly recover $\Gamma_D$ from the divisor boundary alone unless we somehow deal with ambiguities such as this. Fortunately, there is still a great deal of rigidity in the relationship between divisor boundaries and their associated graphs.

In \cite{Neumann:calculus}, the author defined a collection of eight moves which may be performed on a plumbing graph $\Gamma$ (see \cite{Neumann:calculus}, Section 2). Each of these moves have associated topological constructions which transform any associated $Y_\Gamma$ into a different plumbing $Y_{\tilde{\Gamma}}$ while preserving the diffeomorphism type.

There are only two moves which are relevant to this manuscript. They are named and defined diagrammatically by:

\textbf{(1) The $\pm 1$ blow-down:} Defined by
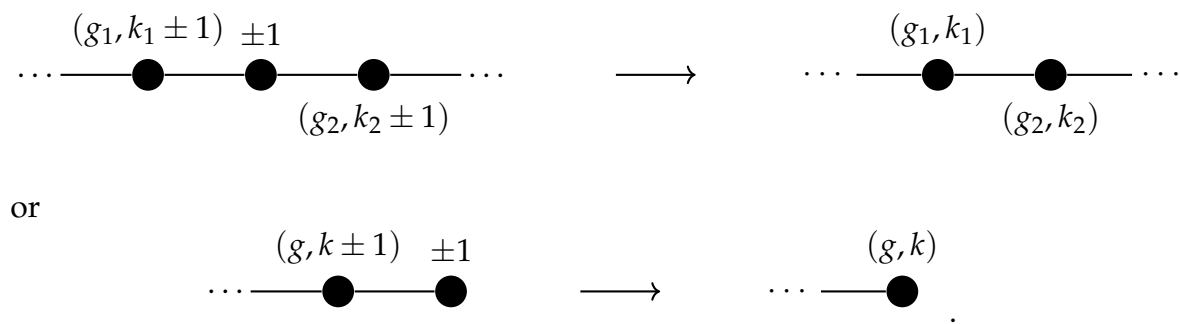
\[
      \begin{tikzpicture}[scale=0.5]
        \GraphInit[vstyle=Empty]
        \SetGraphUnit{3}
        \Vertex[L=$\ldots$]{A}
        \EA[LabelOut, Lpos=90, L={$(g_1,k_1 \pm 1)$}](A){B}
        \EA[LabelOut, Lpos=90, L=$\pm1$](B){BB}
        \EA[LabelOut, Lpos=270, L={$(g_2,k_2 \pm 1)$}](BB){C}
        \EA[L=$\ldots$](C){D}
  
        \EA[NoLabel](D){E}
        \EA[NoLabel](E){F}
        
        \EA[L=$\cdots$](F){G}
        \EA[LabelOut, Lpos=90, L={$(g_1,k_1)$}](G){H}
        \EA[LabelOut, Lpos=270, L={$(g_2,k_2)$}](H){I}
        \EA[L = $\cdots$](I){J}

        \Edges(A,B,BB,C,D)
        \Edges(G,H,I,J)
        \Edge[style={->}](E)(F)
  
        \AddVertexColor{black}{B,BB,C,H,I}
      \end{tikzpicture}
    \]
    or
    \[
      \begin{tikzpicture}[scale=0.5]
        \GraphInit[vstyle=Empty]
        \SetGraphUnit{3}
        \Vertex[L=$\ldots$]{A}
        \EA[LabelOut, Lpos=90, L={$(g,k \pm 1)$}](A){B}
        \EA[LabelOut, Lpos=90, L=$\pm 1$](B){BB}
  
        \EA[NoLabel](BB){C}
        \EA[NoLabel](C){D}
        
        \EA[L=$\cdots$](D){E}
        \EA[LabelOut, Lpos=90, L={$(g,k)$}](E){F}

        \Edges(A,B,BB)
        \Edges(E,F)
        \Edge[style={->}](C)(D)
  
        \AddVertexColor{black}{B,BB,F}
      \end{tikzpicture}.
    \]
and
\textbf{(2) $\RP^2$-absorption:} Defined by
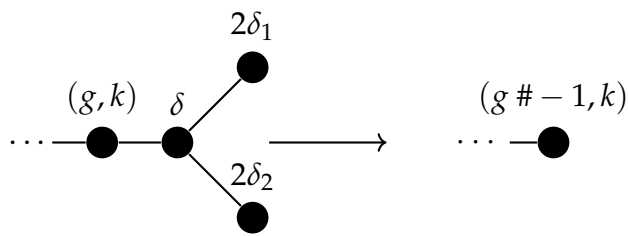
\[
      \begin{tikzpicture}[scale=0.5]
        \GraphInit[vstyle=Empty]
        \SetGraphUnit{2}
        \Vertex[L=$\ldots$]{A}
        \EA[LabelOut, Lpos=90, L={$(g,k)$}](A){B}
        \EA[LabelOut, Lpos=90, L={$\delta$}](B){C}
        \NOEA[LabelOut, Lpos=90, L={$2\delta_1$}](C){D}
        \SOEA[LabelOut, Lpos=90, L={$2\delta_2$}](C){E}
  
        \EA[NoLabel](C){F}
        \EA[NoLabel](F){G}
        \EA[NoLabel](G){H}
        
        \EA[L=$\cdots$](H){I}
        \EA[LabelOut, Lpos=90, L={$(g\ \# -1, k)$}](I){J}

        \Edges(A,B,C,D)
        \Edges(C,E)
        \Edges(I,J)
        \Edge[style={->}](F)(H)
  
        \AddVertexColor{black}{B,C,D,E,J}
      \end{tikzpicture}.
    \] 
Here $\delta_i = \pm 1$, $\delta = \frac{\delta_1 + \delta_2}{2}$, and the vertex with decoration $(g\ \# -1, k)$ denotes the 
$S^1$ bundle over $\Sigma_{g} \# \RP^2$ with Euler number 
$k$.

The other moves are (by name only):
\begin{enumerate}[label=(\arabic*)]\addtocounter{enumi}{2}
    \item $0$-curve absorption: This move obviously implies that the graph contains a sphere of self intersection number zero. These graphs will be discussed in the contact setting below.
    \item Unoriented handle absorption: This move is not applicable to $SNC^+$ divisors as it involves reducing to a graph with a non-orientable vertex.
    \item Oriented handle absorption: This move is equivalent to having a sphere of self-intersection number zero. This situation is dealt with in the contact setting below. 
    \item Splitting: All of our divisor graphs will be connected and do not admit any splittings.
    \item Seifert graph exchanges: This performs a number of graph exchanges between components with one vertex and one loop and the standard star-shaped presentations of the Seifert-fibered spaces they represent. Because none of our graphs contain any loops (i.e. edges from a vertex to itself), these exchanges will never be applicable.
    \item Annulus absorption: This move is performed on general plumbings with boundary. Since none of our divisor boundaries have boundaries of their own, this move does not apply.
\end{enumerate}

\begin{defn}
  Two plumbing graphs $\Gamma_1,\Gamma_2$ are said to be 
  \emph{TPC-combinatorially related}\footnote{TPC = Topological Plumbing Calculus} if we can obtain one from the other via 
  a sequence of moves in Neumann's plumbing calculus (see \cite{Neumann:calculus}, Section 2).
\end{defn}
Which brings us to one of the main results of Neumann's paper:
\begin{thm}[\textbf{\cite{Neumann:calculus}, Theorem 3.1}]\label{thm:neumann}
  If $Y_{\Gamma}$ is diffeomorphic to $Y_{\widetilde{\Gamma}}$, then $\Gamma$ and $\widetilde{\Gamma}$ are TPC-combinatorially related.
\end{thm}
From this, we may conclude that the diffeomorphism type of a divisor boundary determines the isomorphism type of its plumbing graph up to topological plumbing calculus. To prove this theorem, Neumann defines what it means for a plumbing graph to be in \emph{normal form} and first shows that every plumbing graph may be put into normal form via topological plumbing calculus. Neumann then shows that the diffeomorphism type of plumbing boundaries whose graphs are in normal form do determine the isomorphism type of its associated graph. In order to write the definition of Neumann's normal form, it helps to discuss a special class of subgraphs relevant to the theorem: chains of spheres.

\subsection{Chains of Spheres}
A set of plumbing graphs important for our discussion are linear plumbings of spheres or ``chains."
  
\begin{defn}
  Let $\Gamma$ be a plumbing graph. A \emph{chain} is a subgraph $C \subset \Gamma$ of the form
    
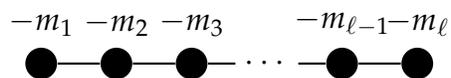
\[\begin{tikzpicture}[scale=0.5]
        \GraphInit[vstyle=Empty]
        \SetGraphUnit{2}

        \Vertex[LabelOut, Lpos=90, L=$-m_1$]{A}
        \EA[LabelOut, Lpos=90, L=$-m_2$](A){B}
        \EA[LabelOut, Lpos=90, L=$-m_3$](B){C}
        \EA[L=$\cdots$](C){D}
        \EA[LabelOut, Lpos=90, L=$-m_{\ell - 1}$](D){E}
        \EA[LabelOut, Lpos=90, L=$-m_\ell$](E){F}
        \AddVertexColor{black}{A,B,C,E,F}
        
        \Edges(A,B,C,D,E,F)
    \end{tikzpicture}\]
    with each vertex having degree at most $2$ in $\Gamma$. The integers $-m_i$ are called 
    \emph{the components of the chain $C$.}
    A chain is called \emph{maximal} if it is not contained in a strictly larger chain. A chain is \emph{normal} or \emph{in normal form} if $m_i \ge 2$ for $i > 1$ and $m_1$ is either $0$ or $\ge 2$.
\end{defn}

\begin{rem}
Chains are symmetric and so $Y_C$ is diffeomorphic to $Y_{C'}$ where $C'$ is the chain with components $(-m_\ell, \ldots, -m_1)$.
\end{rem}

The following lemma is a classical result in 3-manifold topology. There are many proofs of this result in the literature, see for example \cite{Symmington:four}.
\begin{lem}\label{lem:lens}
  Let $C \subset \Gamma$ be a chain with components $(-m_1,\ldots,-m_\ell)$ and $m_\ell \neq 0,1$. Then $Y_{C}$ is diffeomorphic to a Lens space $L(p,q)$ where 
  \begin{align}\label{eqn:continued_fraction}
    -\frac{p}{q} = m_1 - \frac{1}{m_2 - \frac{1}{\ddots - \frac{1}{m_\ell}}}.
  \end{align}
\end{lem}

From now on, we introduce the following notational conventions: 
\begin{itemize}
  \item We let $[m_1,\ldots,m_\ell]$ denote the finite continued fraction with components $(m_1,\ldots,m_\ell)$ as in the right-hand side of \ref{eqn:continued_fraction}.
  \item We let $C(m_1,\ldots, m_{\ell})$ the chain with components $(-m_1,\ldots,-m_n)$. 
  \item We let $L(m_1,\ldots, m_\ell)$ denote the lens space $L(p,q)$ with $-p/q = [m_1,\ldots, m_{\ell}]$.
\end{itemize}


Each positive rational number $r$ has a unique normal chain $C(m_1,\ldots,m_{\ell})$ 
associated to it where $r = [m_1,\ldots, m_{\ell}]$. This follows 
directly from the fact that continued fraction expansions of rational numbers are unique 
if all the components $m_i$ for $i > 1$ satisfy $m_i \ge 2$. Thus every normalized chain 
is homemorphic to a lens space $L(p,q)$ where $p,q$ are relatively prime and $-p/q = r$.
\subsection{TPC normal form}

We are now able define Neumann's normal form.
\begin{defn}[\textbf{\cite{Neumann:calculus}, Section 4}]
  We say that a plumbing graph $\Gamma$ is \emph{in TPC normal form} if the following criteria are met:
  \begin{enumerate}
    \item No operation from topological plumbing calculus may be performed, except that $\Gamma$ may have a component 
      of the form:
        
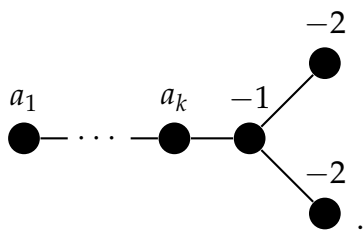
\[
          \begin{tikzpicture}[scale=0.5]
            \GraphInit[vstyle=Empty]
            \SetGraphUnit{2}

            \Vertex[LabelOut, Lpos=90, L=$a_1$]{A}
            \EA[L=$\cdots$](A){B}
            \EA[LabelOut, Lpos=90, L=$a_k$](B){C}
            \EA[LabelOut, Lpos=90, L=$-1$](C){D}
            \NOEA[LabelOut, Lpos=90, L=$-2$](D){E}
            \SOEA[LabelOut, Lpos=90, L=$-2$](D){F}
            \AddVertexColor{black}{A,C,D,E,F}
        
            \Edges(A,B,C,D,E)
            \Edges(D,F)
          \end{tikzpicture}.
        \]
    
    \item The weights $e_i$ on all chains of $\Gamma$ satisfy $e_i \le -2$.
    
    \item No portion of the graph has the from
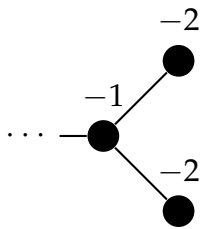
      \[
        \begin{tikzpicture}[scale=0.5]
          \GraphInit[vstyle=Empty]
          \SetGraphUnit{2}

          \Vertex[L=$\cdots$]{A}
          \EA[LabelOut, Lpos=90, L=$-1$](A){D}
          \NOEA[LabelOut, Lpos=90, L=$-2$](D){E}
          \SOEA[LabelOut, Lpos=90, L=$-2$](D){F}
          \AddVertexColor{black}{D,E,F}
      
          \Edges(A,D,E)
          \Edges(D,F)
        \end{tikzpicture}
      \]
      unless it is in a component of $\Gamma$ of the form
      
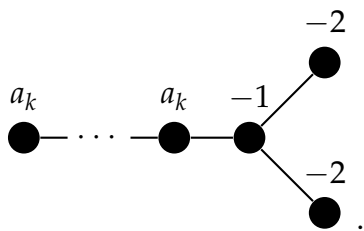
\[
        \begin{tikzpicture}[scale=0.5]
          \GraphInit[vstyle=Empty]
          \SetGraphUnit{2}

          \Vertex[LabelOut, Lpos=90, L=$a_k$]{A}
          \EA[L=$\cdots$](A){B}
          \EA[LabelOut, Lpos=90, L=$a_k$](B){C}
          \EA[LabelOut, Lpos=90, L=$-1$](C){D}
          \NOEA[LabelOut, Lpos=90, L=$-2$](D){E}
          \SOEA[LabelOut, Lpos=90, L=$-2$](D){F}
          \AddVertexColor{black}{A,C,D,E,F}
      
          \Edges(A,B,C,D,E)
          \Edges(D,F)
        \end{tikzpicture}.
      \]
    \item No component of $\Gamma$ is isomorphic to:
      
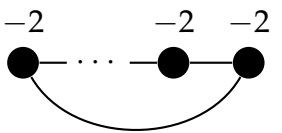
\[
        \begin{tikzpicture}[scale=0.5]
          \GraphInit[vstyle=Empty]
          \SetGraphUnit{2}

          \Vertex[LabelOut, Lpos=90, L=$-2$]{A}
          \EA[L=$\cdots$](A){B}
          \EA[LabelOut, Lpos=90, L=$-2$](B){C}
          \EA[LabelOut, Lpos=90, L=$-2$](C){D}
          \AddVertexColor{black}{A,C,D}

          \Edges(A,B,C,D)
          \tikzset{EdgeStyle/.style = {bend right=60}}
          \Edge(A)(D)
        \end{tikzpicture}.
      \]
  \end{enumerate}
\end{defn}

In the context of divisor graphs, (4) is a special case of a \emph{circular spherical divisor graph}. Such divisors and their graphs were classified in their entirety in \cite{LM:local} and are not relevant to this paper. 

Neumann's major results about graphs in normal form result in a proof of
\begin{thm}[\textbf{\cite{Neumann:calculus}, Theorem 3.1}]\label{thm:neumann}
  If $Y_{D}$ is diffeomorphic to $Y_{\widetilde{D}}$, then $\Gamma_{D}$ and $\Gamma_{\widetilde{D}}$ are TPC-combinatorially related.
\end{thm}

It is first proven that:
\begin{lem}[\textbf{\cite{Neumann:calculus}, Theorem 4.1}]\label{lem:reduction}
  Any plumbing graph can be reduced to normal form using topological plumbing calculus.
\end{lem}
One then only needs to show:
\begin{thm}[\textbf{\cite{Neumann:calculus}, Theorem 4.2 }]\label{thm:main_neumann}
Let $\Gamma,\widetilde{\Gamma}$ be two plumbing graphs in topological normal form. Then $Y_{\Gamma}$ is diffeomorphic to $Y_{\widetilde{\Gamma}}$ if 
and only if $\Gamma$ and $\widetilde{\Gamma}$ are isomorphic.
\end{thm}

To prove \Cref{lem:reduction}, Neumann outlines a procedure for using the topological plumbing calculus to normalize a plumbing graph. Crucial to this procedure is an algorithm for normalizing chains. As we will see in the next section, this algorithm cannot be followed directly in the contact setting while preserving the contactomorphism type of the divisor boundary.

\section{Contact Plumbing Calculus}\label{sec:cpc}

We will see that not every move in the topological plumbing calculus respects the contact topology of the boundary. After determining which moves in the calculus are capable of being carried out in the contact setting, we will augment the calculus with new moves. These moves indeed are already possible within Neumann's calculus but nonetheless have their own unique usage which warrants their own names.

\subsection{Contact-Sensitive TPC Moves}\label{sec:tpc_moves}

\Cref{thm:neumann} tells us that if $Y_D$ and $Y_{\tilde{D}}$ are diffeomorphic, then their divisor graphs are TPC-related. Since contactomorphism implies diffeomorphism, this result clearly also holds with respect to contactomorphism. The issue is that, while all associated topological constructions in Neumann's topological plumbing calculus preserve the diffeomorphism type of the divisor boundary, the moves may change the contactomorphism type of a divisor boundary. An example of a move that is not contact-preserving is the \emph{interior/exterior $+1$ blow-up/blow-down} 
(\cite{Neumann:calculus}, Move 1) defined by 
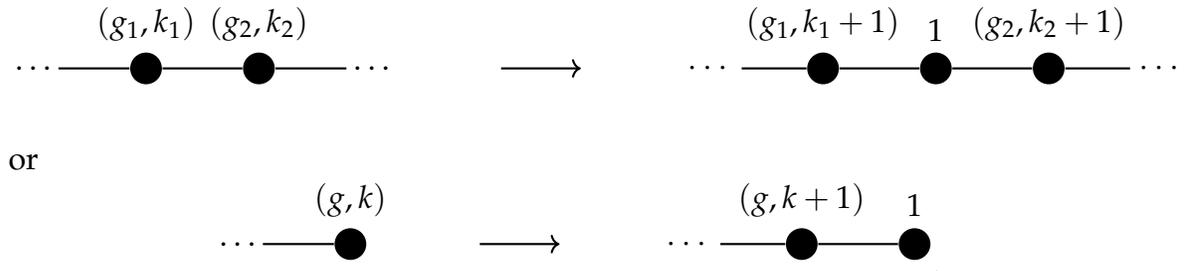
\[
      \begin{tikzpicture}[scale=0.5]
        \GraphInit[vstyle=Empty]
        \SetGraphUnit{3}
        \Vertex[L=$\ldots$]{A}
        \EA[LabelOut, Lpos=90, L={$(g_1,k_1)$}](A){B}
        \EA[LabelOut, Lpos=90, L={$(g_2,k_2)$}](B){C}
        \EA[L=$\ldots$](C){D}
  
        \EA[NoLabel](D){E}
        \EA[NoLabel](E){F}
        
        \EA[L=$\cdots$](F){G}
        \EA[LabelOut, Lpos=90, L={$(g_1,k_1+1)$}](G){H}
        \EA[LabelOut, Lpos=90, L=$1$](H){I}
        \EA[LabelOut, Lpos=90, L={$(g_2,k_2+1)$}](I){J}
        \EA[L = $\cdots$](J){K}

        \Edges(A,B,C,D)
        \Edges(G,H,I,J,K)
        \Edge[style={->}](E)(F)
  
        \AddVertexColor{black}{B,C,H,I,J}
      \end{tikzpicture}
    \]
    or
    \[
      \begin{tikzpicture}[scale=0.5]
        \GraphInit[vstyle=Empty]
        \SetGraphUnit{3}
        \Vertex[L=$\ldots$]{A}
        \EA[LabelOut, Lpos=90, L={$(g,k)$}](A){B}
  
        \EA[NoLabel](B){C}
        \EA[NoLabel](C){D}
        
        \EA[L=$\cdots$](D){E}
        \EA[LabelOut, Lpos=90, L={$(g,k+1)$}](E){F}
        \EA[LabelOut, Lpos=90, L=$1$](F){G}

        \Edges(A,B)
        \Edges(E,F,G)
        \Edge[style={->}](C)(D)
  
        \AddVertexColor{black}{B,F,G}
      \end{tikzpicture}.
    \]
The following example demonstrating this was first given in \cite{LM:capping} Example 2.21:
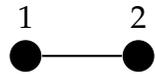
\[
  \begin{tikzpicture}[scale=0.5]
    \GraphInit[vstyle=Empty]
    \SetGraphUnit{3}
    \Vertex[LabelOut, Lpos=90,L=$1$]{A}
    \EA[LabelOut, Lpos=90, L={$2$}](A){B}
    
    \Edges(A,B)
    \AddVertexColor{black}{A,B}
  \end{tikzpicture}
\]
which is the $+1$ blow-up of 
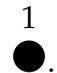
\[
  \begin{tikzpicture}[scale=0.5]
    \GraphInit[vstyle=Empty]
    \SetGraphUnit{3}
    \Vertex[LabelOut, Lpos=90,L=$1$]{A}
    
    \AddVertexColor{black}{A}
  \end{tikzpicture}.
\]
The contact 3-manifold associated to the second graph is $(S^3,\xi_{std})$ but 
the contact structure associated to the first is overtwisted. 

The \emph{interior/exterior $-1$ blow-up/blow-down} defined by
    \[
      \begin{tikzpicture}[scale=0.5]
        \GraphInit[vstyle=Empty]
        \SetGraphUnit{3}
        \Vertex[L=$\ldots$]{A}
        \EA[LabelOut, Lpos=90, L={$(g_1,k_1)$}](A){B}
        \EA[LabelOut, Lpos=90, L={$(g_2,k_2)$}](B){C}
        \EA[L=$\ldots$](C){D}
  
        \EA[NoLabel](D){E}
        \EA[NoLabel](E){F}
        
        \EA[L=$\cdots$](F){G}
        \EA[LabelOut, Lpos=90, L={$(g_1,k_1-1)$}](G){H}
        \EA[LabelOut, Lpos=90, L=$-1$](H){I}
        \EA[LabelOut, Lpos=90, L={$(g_2,k_2-1)$}](I){J}
        \EA[L = $\cdots$](J){K}

        \Edges(A,B,C,D)
        \Edges(G,H,I,J,K)
        \Edge[style={->}](E)(F)
  
        \AddVertexColor{black}{B,C,H,I,J}
      \end{tikzpicture}
    \]

    \[
      \begin{tikzpicture}[scale=0.5]
        \GraphInit[vstyle=Empty]
        \SetGraphUnit{3}
        \Vertex[L=$\ldots$]{A}
        \EA[LabelOut, Lpos=90, L={$(g,k)$}](A){B}
  
        \EA[NoLabel](B){C}
        \EA[NoLabel](C){D}
        
        \EA[L=$\cdots$](D){E}
        \EA[LabelOut, Lpos=90, L={$(g,k-1)$}](E){F}
        \EA[LabelOut, Lpos=90, L=$-1$](F){G}

        \Edges(A,B)
        \Edges(E,F,G)
        \Edge[style={->}](C)(D)
  
        \AddVertexColor{black}{B,F,G}
      \end{tikzpicture}
    \]
is contact-preserving. This was shown in \cite{LM:local} using methods from 
toric geometry. Their construction exploits the toric structure of a neighborhood of the intersection of two smooth divisors. A toric diagram for a portion of a concave neighborhood can be obtained near the intersection points and the -1 blow-up can be performed in the toric region by modifying this diagram. One can then glue in the 4-manifold associated to the modified diagram by cutting out the toric region in the original neighborhood. This can be done while preserving concavity to yield a concave neighborhood of the symplectic divisor obtained after performing the blow-up.

Additionally, many 
of the moves outlined in topological plumbing calculus are not-applicable as they involve vertices with non-orientable bases which have no direct symplectic analog.

Every component of an $SNC^+$ divisor is orientable and so any of the moves involving non-orientable surfaces (\cite{Neumann:calculus}, Moves 2 and 4) can be completely 
removed from consideration with the exception \emph{$\RP^2$-absorption}:
\[
      \begin{tikzpicture}[scale=0.5]
        \GraphInit[vstyle=Empty]
        \SetGraphUnit{2}
        \Vertex[L=$\ldots$]{A}
        \EA[LabelOut, Lpos=90, L={$(g,k)$}](A){B}
        \EA[LabelOut, Lpos=90, L={$\delta$}](B){C}
        \NOEA[LabelOut, Lpos=90, L={$2\delta_1$}](C){D}
        \SOEA[LabelOut, Lpos=90, L={$2\delta_2$}](C){E}
  
        \EA[NoLabel](C){F}
        \EA[NoLabel](F){G}
        \EA[NoLabel](G){H}
        
        \EA[L=$\cdots$](H){I}
        \EA[LabelOut, Lpos=90, L={$(g\ \# -1, k)$}](I){J}

        \Edges(A,B,C,D)
        \Edges(C,E)
        \Edges(I,J)
        \Edge[style={->}](F)(H)
  
        \AddVertexColor{black}{B,C,D,E,J}
      \end{tikzpicture}.
    \] 
Here $\delta_i = \pm 1$, $\delta = \frac{\delta_1 + \delta_2}{2}$, and the vertex with decoration $(g\ \# -1, k)$ denotes the 
$S^1$ bundle over $\Sigma_{g} \# \RP^2$ with Euler number 
$k$. This move has no clear symplectic analog but subtlety appears in this story due to the simple fact that the 
manifold obtained from
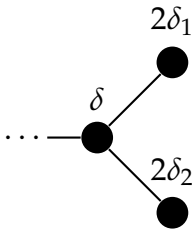
\[
      \begin{tikzpicture}[scale=0.5]
        \GraphInit[vstyle=Empty]
        \SetGraphUnit{2}
        \Vertex[L=$\ldots$]{A}
        \EA[LabelOut, Lpos=90, L={$\delta$}](A){B}
  
        \NOEA[LabelOut, Lpos=90, L={$2\delta_1$}](B){C}
        \SOEA[LabelOut, Lpos=90, L={$2\delta_2$}](B){D}

        \Edges(A,B,C)
        \Edges(B,D)
  
        \AddVertexColor{black}{B,C,D}
      \end{tikzpicture}
    \]
    by cutting at the plumbing torus corresponding to the leftmost edge is 
    diffeomorphic to an orientable tubular neighborhood of a Klein bottle. This subtlety will 
    be expanded upon in \Cref{sec:klein}.

    Edges from a vertex to itself never appear in $SNC^+$ plumbing graphs by the \emph{simplicity} condition (i.e. all the 
    components of $D$ must be smooth manifolds). This allows us to discard any moves involving loops (\cite{Neumann:calculus}, Move 7). 
    Since all intersections of $SNC^+$ divisors are positive, edge signs do not appear and so any moves involving edge signs 
    can be discarded as well (\cite{Neumann:calculus}, Moves 3,5, and 7). Finally, all of our contact 3-manifolds are without boundary 
    so we may discard any moves dealing with graphs with non-empty boundary (Moves 6 and 8). Thus, the only move from Neumann's calculus that is applicable and contact-preserving is the 
    $-1$ blow-up/blow-down.

\subsection{Contact Plumbing Calculus}\label{sec:cpc_description}

The discussion in the previous section tells us that we are disarmed considerably in the contact setting. This makes proving 
an analog of \Cref{lem:reduction} more difficult since it is unclear whether one can reduce divisor graphs to their normal form with $-1$ blow-ups alone. 

Fortunately, since many of Neumann's moves are not-applicable, we merely have 
to avoid using the $+1$ blow-up/blow-down or $\RP^2$-absorptions to reduce to normal form. The proof of \Cref{lem:reduction} (\cite{Neumann:calculus}, Theorem 4.1) crucially uses $+1$ blow-downs to normalize chains and so we must find another way to normalize chains in a way that preserves the contact structure of the divisor boundary. It turns out that $-1$ blow-ups are enough to get pretty far in the normalization process and draw some conclusions.

It is convenient to introduce a few additional contact-preserving moves 
that we can work with. These moves follow directly from applications of $-1$ blow-ups and blow-downs but it is convenient to give them their own names:
\begin{con}[0-curve transfer]
  Consider the following situation
  
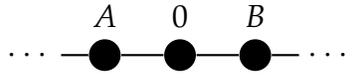
\[\begin{tikzpicture}[scale=0.5]
    \GraphInit[vstyle=Empty]
    \SetGraphUnit{2}
    \Vertex[L=$\cdots$]{A}
    \EA[LabelOut, Lpos=90, L=$A$](A){B}
    \EA[LabelOut, Lpos=90, L=$0$](B){C}
    \EA[LabelOut, Lpos=90, L=$B$](C){D}
    \EA[L=$\cdots$](D){E}
    \AddVertexColor{black}{B,C,D}
    
    \Edges(A,B,C,D,E)
\end{tikzpicture}.\]
We may blow up either edge connected to the $0$-curve. For example, we may blow up the left edge
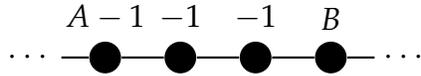
\[\begin{tikzpicture}[scale=0.5]
  \GraphInit[vstyle=Empty]
  \SetGraphUnit{2}
  \Vertex[L=$\cdots$]{A}
  \EA[LabelOut, Lpos=90, L=$A-1$](A){B}
  \EA[LabelOut, Lpos=90, L=$-1$](B){C}
  \EA[LabelOut, Lpos=90, L=$-1$](C){D}
  \EA[LabelOut, Lpos=90, L=$B$](D){E}
  \EA[L=$\cdots$](E){F}
  \AddVertexColor{black}{B,C,D,E}
  
  \Edges(A,B,C,D,E,F)
\end{tikzpicture}\]
then blow down the $(-1)$-curve next to the $B$-curve to arrive at
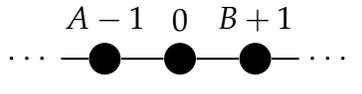
\[\begin{tikzpicture}[scale=0.5]
  \GraphInit[vstyle=Empty]
  \SetGraphUnit{2}
  \Vertex[L=$\cdots$]{A}
  \EA[LabelOut, Lpos=90, L=$A-1$](A){B}
  \EA[LabelOut, Lpos=90, L=$0$](B){C}
  \EA[LabelOut, Lpos=90, L=$B+1$](C){D}
  \EA[L=$\cdots$](D){E}
  \AddVertexColor{black}{B,C,D}
  
  \Edges(A,B,C,D,E)
\end{tikzpicture}.\]
Effectively, we have taken from the $A$-curve and added to the $B$-curve. This process is called \emph{0-curve transfer}. It's 
clear that a similar process may be used to transfer from the $B$-curve to the $A$-curve. This process can be iterated to increase 
the transfer between the chains. We will denote transfers with a directed edge
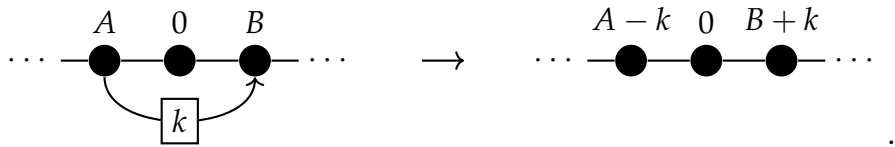
\[\begin{tikzpicture}[scale=0.5]
  \GraphInit[vstyle=Empty]
  \SetGraphUnit{2}

  \Vertex[L=$\cdots$]{A}
  \EA[LabelOut, Lpos=90, L=$A$](A){B}
  \EA[LabelOut, Lpos=90, L=$0$](B){C}
  \EA[LabelOut, Lpos=90, L=$B$](C){D}
  \EA[L=$\cdots$](D){E}

  \EA[NoLabel](E){F}
  \EA[NoLabel](F){G}

  \EA[L=$\cdots$](G){H}
  \EA[LabelOut, Lpos=90, L=$A - k$](H){I}
  \EA[LabelOut, Lpos=90, L=$0$](I){J}
  \EA[LabelOut, Lpos=90, L=$B + k$](J){K}
  \EA[L=$\cdots$](K){L}
  \AddVertexColor{black}{B,C,D,I,J,K}
  
  \Edges(A,B,C,D,E)
  \Edges(H,I,J,K,L)
  \Edge[style={->}](F)(G)
  \Edge[labelstyle={draw},label=$k$,style={bend right,out=-90,in=-90, ->,color=white}](B)(D)
\end{tikzpicture}.\]

A special case is the following
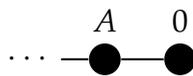
\[\begin{tikzpicture}[scale=0.5]
  \GraphInit[vstyle=Empty]
  \SetGraphUnit{2}
  \Vertex[L=$\cdots$]{A}
  \EA[LabelOut, Lpos=90, L=$A$](A){B}
  \EA[LabelOut, Lpos=90, L=$0$](B){C}
  \EA[NoLabel](C){D}
  \AddVertexColor{black}{B,C}
  
  \Edges(A,B,C)
\end{tikzpicture}\]
in which we may blow up the internal edge
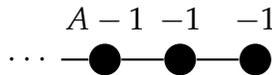
\[\begin{tikzpicture}[scale=0.5]
  \GraphInit[vstyle=Empty]
  \SetGraphUnit{2}
  \Vertex[L=$\cdots$]{A}
  \EA[LabelOut, Lpos=90, L=$A-1$](A){B}
  \EA[LabelOut, Lpos=90, L=$-1$](B){C}
  \EA[LabelOut, Lpos=90, L=$-1$](C){D}
  \AddVertexColor{black}{B,C,D}
  
  \Edges(A,B,C,D)
\end{tikzpicture}\]
and blow down to
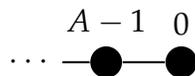
\[\begin{tikzpicture}[scale=0.5]
  \GraphInit[vstyle=Empty]
  \SetGraphUnit{2}
  \Vertex[L=$\cdots$]{A}
  \EA[LabelOut, Lpos=90, L=$A-1$](A){B}
  \EA[LabelOut, Lpos=90, L=$0$](B){C}
  \AddVertexColor{black}{B,C}
  
  \Edges(A,B,C)
\end{tikzpicture}\]
or we may blow up the exterior
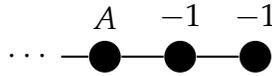
\[\begin{tikzpicture}[scale=0.5]
  \GraphInit[vstyle=Empty]
  \SetGraphUnit{2}
  \Vertex[L=$\cdots$]{A}
  \EA[LabelOut, Lpos=90, L=$A$](A){B}
  \EA[LabelOut, Lpos=90, L=$-1$](B){C}
  \EA[LabelOut, Lpos=90, L=$-1$](C){D}
  \AddVertexColor{black}{B,C,D}
  
  \Edges(A,B,C,D)
\end{tikzpicture}\]
and blow down to 
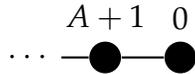
\[\begin{tikzpicture}[scale=0.5]
  \GraphInit[vstyle=Empty]
  \SetGraphUnit{2}
  \Vertex[L=$\cdots$]{A}
  \EA[LabelOut, Lpos=90, L=$A+1$](A){B}
  \EA[LabelOut, Lpos=90, L=$0$](B){C}
  \AddVertexColor{black}{B,C}
  
  \Edges(A,B,C)
\end{tikzpicture}.\]
This is referred to as \emph{transferring to/from nowhere} and will be denoted by
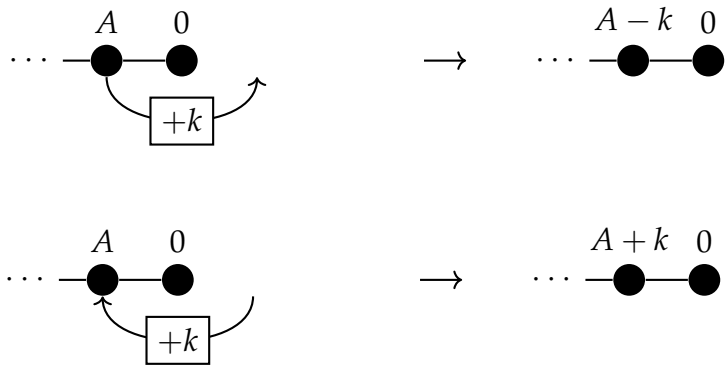
\[\begin{tikzpicture}[scale=0.5]
  \GraphInit[vstyle=Empty]
  \SetGraphUnit{2}

  \Vertex[L=$\cdots$]{A}
  \EA[LabelOut, Lpos=90, L=$A$](A){B}
  \EA[LabelOut, Lpos=90, L=$0$](B){C}
  \EA[NoLabel](C){D}

  \EA[NoLabel](E){F}
  \EA[NoLabel](F){G}

  \EA[L=$\cdots$](G){H}
  \EA[LabelOut, Lpos=90, L=$A - k$](H){I}
  \EA[LabelOut, Lpos=90, L=$0$](I){J}
  \EA[NoLabel](J){K}
  \AddVertexColor{black}{B,C,I,J}
  
  \Edges(A,B,C)
  \Edges(H,I,J)
  \Edge[style={->}](F)(G)
  \Edge[labelstyle={draw},label=$+k$,style={bend right,out=-90,in=-90, ->,color=white}](B)(D)
\end{tikzpicture}\]

\[\begin{tikzpicture}[scale=0.5]
  \GraphInit[vstyle=Empty]
  \SetGraphUnit{2}

  \Vertex[L=$\cdots$]{A}
  \EA[LabelOut, Lpos=90, L=$A$](A){B}
  \EA[LabelOut, Lpos=90, L=$0$](B){C}
  \EA[NoLabel](C){D}

  \EA[NoLabel](E){F}
  \EA[NoLabel](F){G}

  \EA[L=$\cdots$](G){H}
  \EA[LabelOut, Lpos=90, L=$A + k$](H){I}
  \EA[LabelOut, Lpos=90, L=$0$](I){J}
  \EA[NoLabel](J){K}
  \AddVertexColor{black}{B,C,I,J}
  
  \Edges(A,B,C)
  \Edges(H,I,J)
  \Edge[style={->}](F)(G)
  \Edge[labelstyle={draw},label=$+k$,style={bend right,out=-90,in=-90, <-,color=white}](B)(D)
\end{tikzpicture}.\]
\end{con}

\begin{con}[Sliding]
  Consider the following situation
  
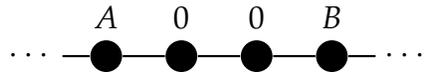
\[\begin{tikzpicture}[scale=0.5]
    \GraphInit[vstyle=Empty]
    \SetGraphUnit{2}
    \Vertex[L=$\cdots$]{A}
    \EA[LabelOut, Lpos=90, L=$A$](A){B}
    \EA[LabelOut, Lpos=90, L=$0$](B){C}
    \EA[LabelOut, Lpos=90, L=$0$](C){D}
    \EA[LabelOut, Lpos=90, L=$B$](D){E}
    \EA[L=$\cdots$](E){F}
    \AddVertexColor{black}{B,C,D,E}
    
    \Edges(A,B,C,D,E,F)
  
  \end{tikzpicture}.\]
  We can use the left $0$-curve in the $0$-$0$ configuration to transfer all of the $A$-curve over. We arrive at
  
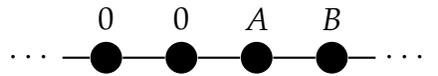
\[\begin{tikzpicture}[scale=0.5]
    \GraphInit[vstyle=Empty]
    \SetGraphUnit{2}
    \Vertex[L=$\cdots$]{A}
    \EA[LabelOut, Lpos=90, L=$0$](A){B}
    \EA[LabelOut, Lpos=90, L=$0$](B){C}
    \EA[LabelOut, Lpos=90, L=$A$](C){D}
    \EA[LabelOut, Lpos=90, L=$B$](D){E}
    \EA[L=$\cdots$](E){F}
    \AddVertexColor{black}{B,C,D,E}
    
    \Edges(A,B,C,D,E,F)

  \end{tikzpicture}.\]
Thus we have moved the $0$-$0$ configuration to the left. This is called \emph{sliding}. It's clear that we can slide in the other direction as well.
\end{con}

\begin{con}[Transfer to/from nowhere]
We start with 
    
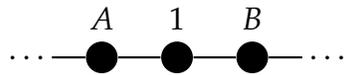
\[
    \begin{tikzpicture}[scale=0.5]
        \GraphInit[vstyle=Empty]
        \SetGraphUnit{2}
        \Vertex[L=$\ldots$]{A}
        \EA[LabelOut, Lpos=90, L=$A$](A){AA}
        \EA[LabelOut, Lpos=90, L=$1$](AA){B}
        \EA[LabelOut, Lpos=90, L=$B$](B){BB}
        \EA[L=$\ldots$](BB){C}
        \AddVertexColor{black}{AA,B,BB}
        \Edges(A,AA,B,BB,C)
        \end{tikzpicture}.
    \]
We blow up once on the left to get
    
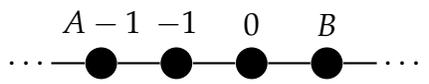
\[
    \begin{tikzpicture}[scale=0.5]
        \GraphInit[vstyle=Empty]
        \SetGraphUnit{2}
        \Vertex[L=$\ldots$]{A}
        \EA[LabelOut, Lpos=90, L=$A-1$](A){AA}
        \EA[LabelOut, Lpos=90, L=$-1$](AA){AAA}
        \EA[LabelOut, Lpos=90, L=$0$](AAA){B}
        \EA[LabelOut, Lpos=90, L=$B$](B){BB}
        \EA[L=$\ldots$](BB){C}
        \AddVertexColor{black}{AA,AAA,B,BB}
        \Edges(A,AA,AAA,B,BB,C)
        \end{tikzpicture}.
    \]
and once on the right to get
    
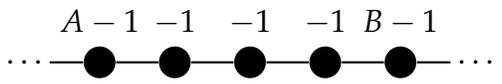
\[
    \begin{tikzpicture}[scale=0.5]
        \GraphInit[vstyle=Empty]
        \SetGraphUnit{2}
        \Vertex[L=$\ldots$]{A}
        \EA[LabelOut, Lpos=90, L=$A-1$](A){AA}
        \EA[LabelOut, Lpos=90, L=$-1$](AA){AAA}
        \EA[LabelOut, Lpos=90, L=$-1$](AAA){B}
        \EA[LabelOut, Lpos=90, L=$-1$](B){BB}
        \EA[LabelOut, Lpos=90, L=$B-1$](BB){BBB}
        \EA[L=$\ldots$](BBB){C}
        \AddVertexColor{black}{AA,AAA,B,BB,BBB}
        \Edges(A,AA,AAA,B,BB,BBB,C)
        \end{tikzpicture}.
    \]
We then blow down the middle curve to arrive at
    
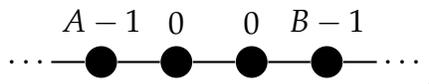
\[
    \begin{tikzpicture}[scale=0.5]
        \GraphInit[vstyle=Empty]
        \SetGraphUnit{2}
        \Vertex[L=$\ldots$]{A}
        \EA[LabelOut, Lpos=90, L=$A-1$](A){AA}
        \EA[LabelOut, Lpos=90, L=$0$](AA){AAA}
        \EA[LabelOut, Lpos=90, L=$0$](AAA){B}
        \EA[LabelOut, Lpos=90, L=$B-1$](B){BB}
        \EA[L=$\ldots$](BB){C}
        \AddVertexColor{black}{AA,AAA,B,BB}
        \Edges(A,AA,AAA,B,BB,C)
        \end{tikzpicture}.
    \]
We refer to this process as \emph{chain replacement} and we denote it by 
    
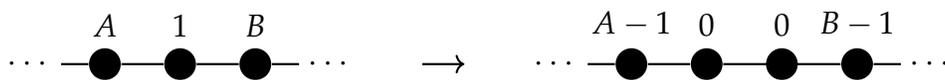
\[
    \begin{tikzpicture}[scale=0.5]
      \GraphInit[vstyle=Empty]
      \SetGraphUnit{2}
    
      \Vertex[L=$\cdots$]{A}
      \EA[LabelOut, Lpos=90, L=$A$](A){B}
      \EA[LabelOut, Lpos=90, L=$1$](B){C}
      \EA[LabelOut, Lpos=90, L=$B$](C){D}
      \EA[L=$\cdots$](D){DD}
    
      \EA[NoLabel](E){F}
      \EA[NoLabel](F){G}
    
      \EA[L=$\cdots$](G){H}E
      \EA[LabelOut, Lpos=90, L=$A-1$](H){I}
      \EA[LabelOut, Lpos=90, L=$0$](I){J}
      \EA[LabelOut, Lpos=90, L=$0$](J){K}
      \EA[LabelOut, Lpos=90, L=$B-1$](K){L}
      \EA[L=$\cdots$](L){M}
      \AddVertexColor{black}{B,C,D,I,J,K,L}
      
      \Edges(A,B,C,D,DD)
      \Edges(H,I,J,L,M)
      \Edge[style={->}](F)(G)
    \end{tikzpicture}
    \]
\end{con}
We will use these moves in order to perform a modified chain reduction procedure.

\subsection{The Modified Chain Reduction Lemma}
We may now prove the following key lemma:
\begin{lem}\label{lem:modified_reduction}
  Let $\Gamma$ be a concave divisor graph. Then we may perform a finite sequence of combinatorial ($-1$)-blow-ups and $(-1)$-blow-downs
  on $\Gamma$ to arrive at a graph $\Gamma^\xi$ whose maximal chains are all of the form 
  
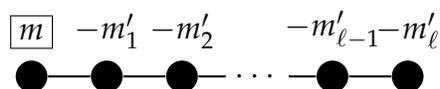
\[\begin{tikzpicture}[scale=0.5]
  \GraphInit[vstyle=Empty]
  \SetGraphUnit{2}

  \Vertex[LabelOut,Lpos=90,L=$\fbox{m}$]{A}
  \EA[LabelOut, Lpos=90, L=$-m_1'$](A){B}
  \EA[LabelOut, Lpos=90, L=$-m_2'$](B){C}
  \EA[L=$\cdots$](C){D}
  \EA[LabelOut, Lpos=90, L=$-m_{\ell - 1}'$](D){E}
  \EA[LabelOut, Lpos=90, L=$-m_\ell'$](E){F}
  \AddVertexColor{black}{A,B,C,E,F}
  
  \Edges(A,B,C,D,E,F)
\end{tikzpicture}\]
    Where the $m$-vertex represents a linear chain of $0$-curves of length $m \ge 0$ and $m_i \ge 2$. Such a chain is said to be in \ul{contact normal form}. We  call $\Gamma^\xi$ the \ul{contact chain reduction of $\Gamma$.}
\end{lem}
\begin{proof}
  Let $C \subset \Gamma$ be a maximal chain. Fix some linear ordering on the components of $C$ so that $C = C(m_{1}, \ldots, m_{\ell})$. For exterior maximal chains, we always take $m_\ell$ to be the decoration on the exterior-most curve. We progress according to the following procedure:
  \begin{enumerate}
  \item Blow down as many $-1$-curves as possible from left to right (possibly in multiple stages). The result should be a maximal chain with no $-1$-curves. If the resulting chain is normalized, proceed to the next chain.
  
  \item For each configuration of the form
  
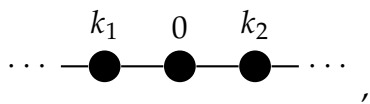
\[\begin{tikzpicture}[scale=0.5]
    \GraphInit[vstyle=Empty]
    \SetGraphUnit{2}

    \Vertex[L=$\cdots$]{A}
    \EA[LabelOut, Lpos=90, L=$k_1$](A){B}
    \EA[LabelOut, Lpos=90, L=$0$](B){C}
    \EA[LabelOut, Lpos=90, L=$k_2$](C){D}
    \EA[L=$\cdots$](D){E}
    \AddVertexColor{black}{B,C,D}
    
    \Edges(A,B,C,D,E)
\end{tikzpicture},\]
transfer all of $k_1$ over the $0$-curve to the right
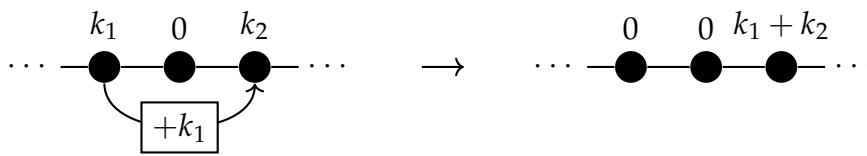
\[\begin{tikzpicture}[scale=0.5]
  \GraphInit[vstyle=Empty]
  \SetGraphUnit{2}

  \Vertex[L=$\cdots$]{A}
  \EA[LabelOut, Lpos=90, L=$k_1$](A){B}
  \EA[LabelOut, Lpos=90, L=$0$](B){C}
  \EA[LabelOut, Lpos=90, L=$k_2$](C){D}
  \EA[L=$\cdots$](D){E}

  \EA[NoLabel](E){F}
  \EA[NoLabel](F){G}

  \EA[L=$\cdots$](G){H}
  \EA[LabelOut, Lpos=90, L=$0$](H){I}
  \EA[LabelOut, Lpos=90, L=$0$](I){J}
  \EA[LabelOut, Lpos=90, L=$k_1 + k_2$](J){K}
  \EA[L=$\cdots$](K){L}
  \AddVertexColor{black}{B,C,D,I,J,K}
  
  \Edges(A,B,C,D,E)
  \Edges(H,I,J,K,L)
  \Edge[style={->}](F)(G)
  \Edge[labelstyle={draw},label=$+k_1$,style={bend right,out=-90,in=-90, ->,color=white}](B)(D)
\end{tikzpicture}.\]

  Now shift the 0-0 configuration all the way to the left. 

  \item Repeat steps (1) and (2), until the chain has no $(-1)$-curves and no $0$-curves except those on the left. Since both steps either preserve or reduce the number of vertices in the chain, this process will eventually halt.

  \item The chain is now of the form:
  
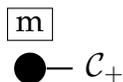
\[\begin{tikzpicture}[scale=0.5]
  \GraphInit[vstyle=Empty]
  \SetGraphUnit{2}

  \Vertex[LabelOut,Lpos=90,L=$\fbox{m}$]{A}
  \EA[L=$\mathcal{C}_+$](A){B}
  \AddVertexColor{black}{A}
  
  \Edges(A,B)
\end{tikzpicture}\]
where the $\fbox{m}$-vertex is a linear chain 0-curves of length $m$ and $\mathcal{C}_+$ is a subchain without $-1$-curves or $0$-curves. Starting from the left, identify the first $k$-curve in the chain with $k > 0$. The chain thus may be decomposed as
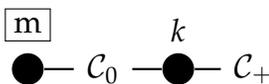
\[\begin{tikzpicture}[scale=0.5]
    \GraphInit[vstyle=Empty]
    \SetGraphUnit{2}
    \Vertex[LabelOut,Lpos=90,L=$\fbox{m}$]{A}
    \EA[L=$\mathcal{C}_0$](A){B}
    \EA[LabelOut, Lpos=90, L=$k$](B){C}
    \EA[L=$\mathcal{C}_+$](C){D}
    \AddVertexColor{black}{A,C}
    
    \Edges(A,B,C,D)
  \end{tikzpicture}\]
  where $\mathcal{C}_0$ is a chain with all decorations $\le -2$ and $\mathcal{C}_+$ is the rest of the chain. 

  We could have $\mathcal{C}_0 = \emptyset$ and be in the situation:
  
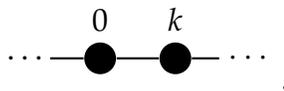
\[\begin{tikzpicture}[scale=0.5]
    \GraphInit[vstyle=Empty]
    \SetGraphUnit{2}
    \Vertex[L=$\ldots$]{A}
    \EA[LabelOut, Lpos=90, L=$0$](A){B}
    \EA[LabelOut, Lpos=90, L=$k$](B){C}
    \EA[L=$\cdots$](C){D}
    \AddVertexColor{black}{B,C}
    
    \Edges(A,B,C,D)
  \end{tikzpicture}.\]
    If the number of zeros on the left is odd, we simply use the chain of $0$-curves to the left to transfer $k$ into the rest of the graph. 
  This increases the number of 0-curves on the left by one.
  
  If the number of 0-curves on the left is even and non-zero, we may transfer $+2$ from 
  the $0$-curve left of the $k$-curve into the graph using the odd number of $0$-curves to its left. This reduces the scenario to the case $\mathcal{C}_0 \neq \emptyset$ and 
  decreases the number of $0$-curves on the left by one.

  In all other cases, we blow up the $k$-curve on the left $(k-1)$ times so that it becomes a zero curve, the chain is now of the form: 
    
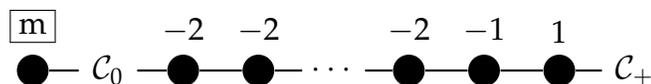
\[\begin{tikzpicture}[scale=0.5]
    \GraphInit[vstyle=Empty]
    \SetGraphUnit{2}
    \Vertex[LabelOut,Lpos=90,L=$\fbox{m}$]{A}
    \EA[L=$\mathcal{C}_0$](A){B}
    \EA[LabelOut, Lpos=90, L=$-2$](B){C}
    \EA[LabelOut, Lpos=90, L=$-2$](C){D}
    \EA[L=$\cdots$](D){E}
    \EA[LabelOut, Lpos=90, L=$-2$](E){F}
    \EA[LabelOut, Lpos=90, L=$-1$](F){G}
    \EA[LabelOut, Lpos=90, L=$1$](G){H}
    \EA[L=$\mathcal{C}_+$](H){I}
    \AddVertexColor{black}{A,C,D,F,G,H}
    
    \Edges(A,B,C,D,E,F,G,H,I)
  \end{tikzpicture}\]
  and this process has added $(k-1)$ vertices. We then perform a chain replacement and replace the $+1$-curve with a 0-0 configuration which adds an additional vertex. If $\mathcal{C}_+ \neq \emptyset$, then this leaves the chain in the form  
     
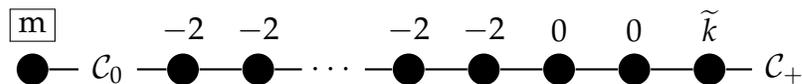
\[\begin{tikzpicture}[scale=0.5]
    \GraphInit[vstyle=Empty]
    \SetGraphUnit{2}
    \Vertex[LabelOut,Lpos=90,L=$\fbox{m}$]{A}
    \EA[L=$\mathcal{C}_0$](A){B}
    \EA[LabelOut, Lpos=90, L=$-2$](B){C}
    \EA[LabelOut, Lpos=90, L=$-2$](C){D}
    \EA[L=$\cdots$](D){E}
    \EA[LabelOut, Lpos=90, L=$-2$](E){F}
    \EA[LabelOut, Lpos=90, L=$-2$](F){G}
    \EA[LabelOut, Lpos=90, L=$0$](G){H}
    \EA[LabelOut, Lpos=90, L=$0$](H){I}
    \EA[LabelOut, Lpos=90, L=$\widetilde{k}$](I){J}
    \EA[L=$\mathcal{C}_+$](J){K}
    \AddVertexColor{black}{A,C,D,F,G,H,I,J}
    
    \Edges(A,B,C,D,E,F,G,H,I,J,K)
  \end{tikzpicture}\]
  where we have 
  \[\widetilde{k} = \begin{cases}
   \le -2 &\\
    -1&\\
    0 &\\
    \ge 1&\\
    \end{cases}.\]
    We may slide the 0-0 configuration so that it is of the form
   
     
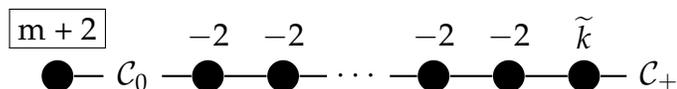
\[\begin{tikzpicture}[scale=0.5]
    \GraphInit[vstyle=Empty]
    \SetGraphUnit{2}
    \Vertex[LabelOut,Lpos=90,L=$\fbox{m + 2}$]{A}
    \EA[L=$\mathcal{C}_0$](A){B}
    \EA[LabelOut, Lpos=90, L=$-2$](B){C}
    \EA[LabelOut, Lpos=90, L=$-2$](C){D}
    \EA[L=$\cdots$](D){E}
    \EA[LabelOut, Lpos=90, L=$-2$](E){F}
    \EA[LabelOut, Lpos=90, L=$-2$](F){G}
    \EA[LabelOut, Lpos=90, L=$\widetilde{k}$](G){H}
    \EA[L=$\mathcal{C}_+$](H){I}
    \AddVertexColor{black}{A,C,D,F,G,H}
    
    \Edges(A,B,C,D,E,F,G,H,I)
  \end{tikzpicture}\]
    If $\mathcal{C}_+ = \emptyset$, then we proceed to the next chain after sliding. It will be important to notate the chain of $-2$'s for the next step. We will represent our chain by 
    \[\begin{tikzpicture}[scale=0.5]
    \GraphInit[vstyle=Empty]
    \SetGraphUnit{2}
    \Vertex[LabelOut,Lpos=90,L=$\fbox{m+2}$]{A}
    \EA[L=$\mathcal{C}_0$](A){B}
    \EA[NoLabel](B){BB}
    \EA[L=\fbox{$(k-1)$}](BB){C}
    \EA[NoLabel](C){CC}
    \EA[LabelOut, Lpos=90, L=$\widetilde{k}$](CC){D}
    \EA[L=$\mathcal{C}_+$](D){E}
    \AddVertexColor{black}{A,D}
    
    \Edges(A,B,C,D,E)
    \end{tikzpicture}.\]
    By considering the $-2$-curves at the tail end of $\mathcal{C}_0$, we may further write
    \[\begin{tikzpicture}[scale=0.5]
    \GraphInit[vstyle=Empty]
    \SetGraphUnit{2}
    \Vertex[LabelOut,Lpos=90,L=$\fbox{m+2}$]{A}
    \EA[L=$\mathcal{C}_-$](A){B}
    \EA[NoLabel](B){BB}
    \EA[L=\fbox{$N$}](BB){C}
    \EA[NoLabel](C){CC}
    \EA[LabelOut, Lpos=90, L=$\widetilde{k}$](CC){D}
    \EA[L=$\mathcal{C}_+$](D){E}
    \AddVertexColor{black}{A,D}
    
    \Edges(A,B,C,D,E)
    \end{tikzpicture}\]
    where $N$ is a maximal subchain of $-2$-curves and $C_-$ is a subchain with only negative vertices and no $-1$-curves.  
  \item

If the algorithm did not terminate at the previous step, then we are in the following situation:
    \[\begin{tikzpicture}[scale=0.5]
    \GraphInit[vstyle=Empty]
    \SetGraphUnit{2}
    \Vertex[LabelOut,Lpos=90,L=$\fbox{M}$]{A}
    \EA[L=$\mathcal{C}_-$](A){B}
    \EA[NoLabel](B){BB}
    \EA[L=\fbox{$N$}](BB){C}
    \EA[NoLabel](C){CC}
    \EA[LabelOut, Lpos=90, L=$\widetilde{k}$](CC){D}
    \EA[L=$\mathcal{C}_+$](D){E}
    \AddVertexColor{black}{A,D}
    
    \Edges(A,B,C,D,E)
    \end{tikzpicture}\]
We consider all cases for the value of $\widetilde{k}$ separately:
  \begin{itemize}
      \item $\widetilde{k} \le -2$ or $\widetilde{k} \ge 1$: For these, we continue to step through the chain or repeat the process outlined in step 4 if necessary.
      \item $\widetilde{k} = 0$: We proceed as in step 2. At the end of this process, the chain is of the form:
          \[\begin{tikzpicture}[scale=0.5]
    \GraphInit[vstyle=Empty]
    \SetGraphUnit{2}
    \Vertex[LabelOut,Lpos=90,L=$\fbox{M + 2}$]{A}
    \EA[L=$\mathcal{C}_-$](A){B}
    \EA[NoLabel](B){BB}
    \EA[L=\fbox{$N-1$}](BB){C}
    \EA[NoLabel](C){CC}
    \EA[L=$\mathcal{C}_+$](CC){D}
    \AddVertexColor{black}{A}
    
    \Edges(A,B,C,D)
    \end{tikzpicture}\]
  where $\mathcal{C}_+$ is the remainder of the chain. If $\mathcal{C}_+$ is not empty, then the situation is like
  \[\begin{tikzpicture}[scale=0.5]
    \GraphInit[vstyle=Empty]
    \SetGraphUnit{2}
    \Vertex[LabelOut,Lpos=90,L=$\fbox{M+2}$]{A}
    \EA[L=$\mathcal{C}_-$](A){B}
    \EA[NoLabel](B){BB}
    \EA[L=\fbox{$N-1$}](BB){C}
    \EA[NoLabel](C){CC}
    \EA[LabelOut, Lpos=90, L=$\widehat{k}$](CC){D}
    \EA[L=$\mathcal{C}_+$](D){E}
    \AddVertexColor{black}{A,D}
    
    \Edges(A,B,C,D,E)
    \end{tikzpicture}\]
    so we may repeat this step for the new value $\widehat{k}$.

    \item $\widetilde{k} = -1:$ We are in the situation
        
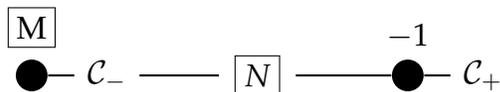
\[\begin{tikzpicture}[scale=0.5]
    \GraphInit[vstyle=Empty]
    \SetGraphUnit{2}
    \Vertex[LabelOut,Lpos=90,L=$\fbox{M}$]{A}
    \EA[L=$\mathcal{C}_-$](A){B}
    \EA[NoLabel](B){BB}
    \EA[L=\fbox{$N$}](BB){C}
    \EA[NoLabel](C){CC}
    \EA[LabelOut, Lpos=90, L=$-1$](CC){D}
    \EA[L=$\mathcal{C}_+$](D){E}
    \AddVertexColor{black}{A,D}
    
    \Edges(A,B,C,D,E)
    \end{tikzpicture}.\]
    If $\mathcal{C}_+ \neq \emptyset$, then we're actually in the situation
    
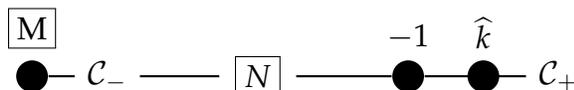
\[\begin{tikzpicture}[scale=0.5]
    \GraphInit[vstyle=Empty]
    \SetGraphUnit{2}
    \Vertex[LabelOut,Lpos=90,L=$\fbox{M}$]{A}
    \EA[L=$\mathcal{C}_-$](A){B}
    \EA[NoLabel](B){BB}
    \EA[L=\fbox{$N$}](BB){C}
    \EA[NoLabel](C){CC}
    \EA[LabelOut, Lpos=90, L=$-1$](CC){D}
    \EA[LabelOut, Lpos=90, L=$\widehat{k}$](D){DD}
    \EA[L=$\mathcal{C}_+$](DD){E}
    \AddVertexColor{black}{A,D,DD}
    
    \Edges(A,B,C,D,E)
    \end{tikzpicture}.\]
    In this case (and even in the case when $\mathcal{C}_+ = \emptyset$) we blow down the $-1$-curve. Because of the chain of $-2$'s to the left of the $-1$-curve, we may continue to blow down until the chain is of the form
        
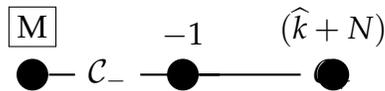
\[\begin{tikzpicture}[scale=0.5]
    \GraphInit[vstyle=Empty]
    \SetGraphUnit{2}
    \Vertex[LabelOut,Lpos=90,L=$\fbox{M}$]{A}
    \EA[L=$\mathcal{C}_-$](A){B}
    \EA[LabelOut, Lpos=90, L=$-1$](B){D}
    \EA[NoLabel](D){DD}
    \EA[LabelOut, Lpos=90, L=$(\widehat{k} + N)$](DD){DDD}
    \EA[L=$\mathcal{C}_+$](DD){E}
    \AddVertexColor{black}{A,D,DDD}
    
    \Edges(A,B,C,D,E)
    \end{tikzpicture}.\]
    Blowing down the remaining -1 curve will affect the rightmost vertex in the subchain $\mathcal{C}_-$. At worst, this vertex is a $-3$-curve so that blowing down leave a trailing $-2$-curve. So, at worst, the situation is 
            
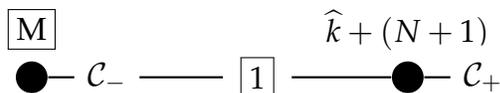
\[\begin{tikzpicture}[scale=0.5]
    \GraphInit[vstyle=Empty]
    \SetGraphUnit{2}
    \Vertex[LabelOut,Lpos=90,L=$\fbox{M}$]{A}
    \EA[L=$\mathcal{C}_-$](A){B}
    \EA[NoLabel](B){BB}
    \EA[L=\fbox{$1$}](BB){C}
    \EA[NoLabel](C){CC}
    \EA[LabelOut, Lpos=90, L=$\widehat{k} + (N+1)$](CC){D}
    \EA[L=$\mathcal{C}_+$](D){E}
    \AddVertexColor{black}{A,D}
    
    \Edges(A,B,C,D,E)
    \end{tikzpicture}.\]
    We continue to step through the chain, starting with the $\widehat{k} + (N+1)$ curve.   
  \end{itemize}

    We continue to repeat the present step until the maximal chain is in contact normal form.

  
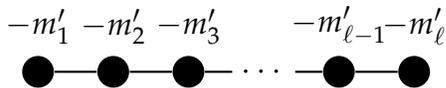
\[\begin{tikzpicture}[scale=0.5]
    \GraphInit[vstyle=Empty]
    \SetGraphUnit{2}

    \Vertex[LabelOut,Lpos=90,L=$-m_1'$]{A}
    \EA[LabelOut, Lpos=90, L=$-m_2'$](A){B}
    \EA[LabelOut, Lpos=90, L=$-m_3'$](B){C}
    \EA[L=$\cdots$](C){D}
    \EA[LabelOut, Lpos=90, L=$-m_{\ell - 1}'$](D){E}
    \EA[LabelOut, Lpos=90, L=$-m_{\ell}'$](E){F}
    \AddVertexColor{black}{A,B,C,E,F}
    
    \Edges(A,B,C,D,E,F)
\end{tikzpicture}\]
or
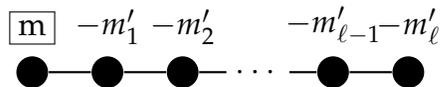
\[\begin{tikzpicture}[scale=0.5]
  \GraphInit[vstyle=Empty]
  \SetGraphUnit{2}

  \Vertex[LabelOut,Lpos=90,L=$\fbox{m}$]{A}
  \EA[LabelOut, Lpos=90, L=$-m_1'$](A){B}
  \EA[LabelOut, Lpos=90, L=$-m_2'$](B){C}
  \EA[L=$\cdots$](C){D}
  \EA[LabelOut, Lpos=90, L=$-m_{\ell - 1}'$](D){E}
  \EA[LabelOut, Lpos=90, L=$-m_\ell'$](E){F}
  \AddVertexColor{black}{A,B,C,E,F}
  
  \Edges(A,B,C,D,E,F)
\end{tikzpicture}\]
where the $\fbox{m}$-vertex is a linear chain of $m$ 0-curves and $m_i \ge 2$.
  \end{enumerate}
  We then repeat the entire procedure for all other chains until all chains are in contact normal form.
\end{proof}

\subsection{Contact normal form plumbing graphs}

The modified chain reduction lemma shows that we can almost fully normalize a chain in the contact setting using only blow-ups and blow-downs. The only difference between a fully normalized chain and a chain in contact normal form is the admissibility of any number of leading 0-curves. This difference, as well as the issues involving Klein bottle pieces discussed earlier, are the only obstructions to carrying out Neumann's reduction procedure in full. 
We make the following definition:
\begin{defn}
    We say that a divisor graph $\Gamma$ is in \emph{contact normal form} if it is in TPC normal form except that:
    \begin{itemize}
        \item $\Gamma$ may contain configurations of the form 
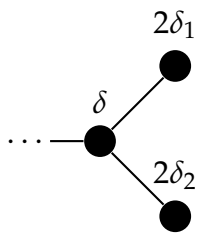
\[
      \begin{tikzpicture}[scale=0.5]
        \GraphInit[vstyle=Empty]
        \SetGraphUnit{2}
        \Vertex[L=$\ldots$]{A}
        \EA[LabelOut, Lpos=90, L={$\delta$}](A){B}
  
        \NOEA[LabelOut, Lpos=90, L={$2\delta_1$}](B){C}
        \SOEA[LabelOut, Lpos=90, L={$2\delta_2$}](B){D}

        \Edges(A,B,C)
        \Edges(B,D)
  
        \AddVertexColor{black}{B,C,D}
      \end{tikzpicture}
    \]
    where $\delta_i = \pm 1$, $\delta = \frac{\delta_1 + \delta_2}{2}$.
        \item A chain $C\subset \Gamma$ may contain any number of 
        leading 0-curves. 
    \end{itemize}
\end{defn}

With this definition and the discussion in the previous section, we have shown
\begin{prop}
    The contact chain reduction $\Gamma^\xi$ of a plumbing graph $\Gamma$ is in contact normal form. 
\end{prop}
We refer to $\Gamma^\xi$ as the \emph{contact reduction} of the graph $\Gamma$.
\begin{cor}
Any plumbing graph $\Gamma$ can be reduced to a graph $\Gamma^{\xi}$ in contact normal form.
\end{cor}

The two distinguising features of a plumbing graph in contact normal form and a plumbing graph in topological normal form are:
\begin{enumerate}[label=(\roman*)]
    \item The plumbing graph may have Klein bottle pieces.
    \item Any chains may contain any number of leading 0-curves.
\end{enumerate}
The first of these conditions has topological implications and the second condition has symplectic and contact topological implications. In the next section, we will analyze the topological ramifications of having an embedded Klein bottle in a divisor boundary. The existence of leading zero chains is closely related to the concept of Giroux torsion from contact geometry. We hope to explicate the details of this in later work. For our applications, it suffices for us to understand the second condition in the context of our particular example (see \Cref{sec:example}). We start with the first condition.

\subsection{Klein bottle pieces}\label{sec:klein}
We turn to the case of plumbing graphs which contain sub-graphs of the form 
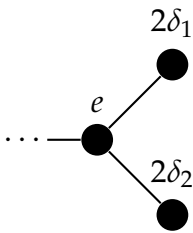
\[
    \begin{tikzpicture}[scale=0.5]
        \GraphInit[vstyle=Empty]
        \SetGraphUnit{2}
        \Vertex[L=$\ldots$]{A}
        \EA[LabelOut, Lpos=90, L={$e$}](A){B}
  
        \NOEA[LabelOut, Lpos=90, L={$2\delta_1$}](B){C}
        \SOEA[LabelOut, Lpos=90, L={$2\delta_2$}](B){D}

        \Edges(A,B,C)
        \Edges(B,D)
  
        \AddVertexColor{black}{B,C,D}
      \end{tikzpicture}
\]
with $\delta_i = \pm 1$. If we cut $Y_\Gamma$ along the plumbing torus corresponding to the leftmost edge of this sub-graph, we split $Y_\Gamma$ into two pieces:
\begin{itemize}
    \item A piece $Y^*_{\Gamma}$ corresponding to the complement of the Klein bottle piece (i.e. the rest of the graph)
    \item A piece $Y_K^*$ which is diffeomorphic to an orientable tubular neighborhood of a Klein bottle $K \subset Y_K^* \subset Y_\Gamma$.
\end{itemize}

To see that $Y_K^*$ indeed has this topology, we observe from the corresponding portion of the plumbing graph that $Y_K^*$ is a Seifert-fibered space over a disk with two singular fibers with rational surgery coefficients $\pm \frac{1}{2}$. This gives a Seifert-fibered presentation of the twisted $I$-bundle over the Klein bottle (which itself is the twisted $S^1$-bundle over $S^1$). (cf. \cite{Hatcher:3manifolds}, Chapter 2.1).

The boundary $\partial Y_K^*$ is diffeomorphic to a torus. We say that $K$ is \emph{virtually compressible} if $\partial Y_K^*$ is compressible, otherwise we say that $K$ is \emph{virtually incompressible} or \emph{virtually essential}. For prime 3-manifolds, admitting an embedded virtually compressible Klein bottle puts severe restrictions on their topology. 

\begin{prop}\label{prop:compressible_klein1}
Suppose that $Y$ is a prime 3-manifold which contains an embedded Klein bottle. If this Klein bottle is virtually compressible, then $Y$ may be given the structure of a Seifert-fibered space over $\RP^2$ with at most one singular fiber.
\end{prop}
\begin{proof}
As above, every Klein bottle piece yields an associated 3-manifold $Y_K^*$ diffeomorphic to an orientable tubular neighborhood of an embedded Klein bottle with $\partial Y_K^*$ diffeomorphic to a torus.

Let $\Delta \subset Y - Y_K^*$ be a compressing disk for $\partial Y_K^*$. The union $\partial Y_K^* \cup \Delta$ is a torus with a meridional disk attached. It follows that the boundary of a smoothing $Y_K$ of the union of $Y_K^*$ and a neighborhood of $\Delta$ is diffeomorphic to a sphere. This allows us to decompose $Y_\Gamma$ as a connected sum with $Y_K \cup B^3$ appearing as a (non-trivial) summand. Since $Y_\Gamma$ was assumed to be prime, we know that the complement $Y - Y_K$ is diffeomorphic to a 3-ball.

It follows that $Y - Y_K^*$ is diffeomorphic to a solid torus and so $Y$ is diffeomorphic to a solid torus glued to an orientable tubular neighborhood of a Klein bottle along some diffeomorphism of their torus boundaries. Since $Y_K^*$ has the structure of an $S^1$-bundle over the M\"{o}bius strip, it follows that $Y$ may be given the structure of a Seifert fibration over $\RP^2$ with at most one singular fiber (c.f. \cite{Hatcher:3manifolds}, Theorem 2.3(d)).
\end{proof}

\begin{cor}\label{cor:essential}
    If $Y_{\Gamma}$ is a prime plumbing boundary which is not Seifert-fibered, then every Klein bottle piece in $Y_{\Gamma}$ contains a virtually essential Klein bottle.
\end{cor}

Thus if a plumbing boundary is not Seifert-fibered, it cannot contain a virtually compressible Klein bottle. 
While this does not fully deal with the issue of Klein bottle pieces, it does give rise to a simple test for eliminating the possibility of some embedded Klein bottles. To fully manage Klein bottle pieces in the contact setting, one would need to understand the contact-topological implications of the existence of virtually incompressible Klein bottles. We hope to pursue this in future work. For our present purposes, we will avoid analyzing this issue entirely by showing that the divisor boundary in our example cannot admit Klein bottles of any kind, virtually compressible or otherwise. This will be expanded upon in \Cref{sec:birational}.

\section{Main Results and Conclusions}\label{sec:birational}

In this section, we will collect a number of topological implications for divisor boundaries, compactifications, and neighborhoods that follow from the results in the previous sections. These results are collected with a goal of focusing on the nature of the structure of all divisor compactifications of a given Liouville 4-manifold $X$. We hope that these results will help contribute to the theory of \emph{birational symplectic geometry} in dimension 4.

In the next section, as an application of these results, we will prove \Cref{thm:main_example} which states that the submanifold $X_{KT} \subset M_{KT}$ does not admit a Stein completion to an affine variety. We will discuss how our arguments may be extended to produce other examples and conclude after briefly mentioning further directions for studying non-affine symplectic manifolds.

\subsection{Structural theorems and topological implications}

Let $(X,d\lambda)$ be a Liouville 4-manifold with finite topology. We will consider the collection of concave $SNC^+$ divisor compactifications of $X$. Recall that such a compactification is a pair $(M,D)$ with $(M,\omega)$ a closed symplectic manifold and $D \subset M$ a concave $SNC^+$ divisor such that $M - D$ is deformation equivalent to $X$. Each such compactification has an associated contact divisor boundary $(Y_D,\xi_D)$ obtained as the contact-type boundary of a concave neighborhood $(N_D,\omega)$ of $D$. Recall that we say that a divisor $D$ is obstructed if $\Gamma^{\xi}_D \not \approx \Gamma^{Top}_D$ and unobstructed otherwise. Our results allow us to characterize the compactifications of obstructed divisors whose associated divisor boundary $Y_D$ is a prime 3-manifold. If this is the case, then we say that $D$ is \emph{prime}.

\begin{thm}\label{thm:main_topological}
Let $(M,D)$ be a concave $SNC^+$ divisor compactification of $X$. Then if $D$ is obstructed and prime, we have one or more of the following:
\begin{enumerate}[label=(\roman*)]
    \item $M$ is a blow-up of a rational or ruled symplectic manifold,
    \item the contact 3-manifold $(\partial_\infty X, \xi)$ is Seifert-fibered over $\RP^2$ via a fibration with at most one singular fiber, or
    \item the contact 3-manifold $(\partial_{\infty} X, \xi)$ contains an embedded incompressible Klein bottle.
\end{enumerate}
\end{thm}
This characterization follows directly from a characterization of concave neighborhoods of obstructed prime divisors.
\begin{proof}[\normalfont \textbf{Theorem~\ref{thm:divisor_boundary}}]
\renewcommand{\qedsymbol}{}
\em{
Let $(N_{D}, \omega)$ be a concave neighborhood of $D \subset M$. If the divisor $D$ is obstructed and prime, then one or more of the following are true:
\begin{enumerate}[label=(\roman*)]
    \item up to blow-ups $(N_D,\omega)$ is diffeomorphic to a concave divisor neighborhood $(N_{\widetilde{D}},\widetilde{\omega})$ whose associated divisor contains a 0-curve $S$ such that $S\cdot S = 0$,
    \item the contact boundary $(\partial N_D, \xi_D)$ is Seifert-fibered over the $\RP^2$ via a fibration with at most one singular fiber, or
    \item the contact boundary $(\partial N_D, \xi_{D})$ contains an embedded incompressible Klein bottle.
\end{enumerate}
}
\end{proof}
\begin{proof}
    Since $D$ is obstructed, $\Gamma^{\xi}_D \neq \Gamma^{Top}_D$. For this to occur, we must have one or both of the following:
    \begin{itemize}
        \item The contact reduction procedure for $\Gamma_D$ produces chains with leading 0-curves,
        \item The contact reduction $\Gamma^{\xi}_D$ contains Klein bottle pieces.
    \end{itemize}
    By following the contact reduction procedure for $\Gamma_D$, we may perform a sequence of symplectic blow-ups on $(N_D,\omega)$ to obtain a concave divisor neighborhood $(N_{\widetilde{D}}, \widetilde{D})$ associated to a concave divisor $\widetilde{D}$ whose associated graph $\Gamma_{\widetilde{D}}$ is isomorphic to the contact reduction $\Gamma_{D}^\xi$. 

    If the first case occurs, then we may conclude that $N_{\widetilde{D}}$ contains a symplectic sphere with self intersection number 0 and so we may conclude (i). If the second case occurs, since $D$ is prime, conclusions (ii) and (iii) follow from \Cref{prop:compressible_klein1} and \Cref{cor:essential}.
\end{proof}

\begin{proof}[\textbf{Proof of \ref{thm:main_topological}}] This result follows after applying a theorem of McDuff for ruled symplectic manifolds \cite[Corollary~1.5(ii)]{mcduff:rational}. In particular, \Cref{thm:main_topological}(i) follows from the fact that a compact symplectic manifold with a symplectic sphere of self-intersection number 0 is the blow-up of a ruled surface (topologically an $S^2$ bundle over a closed surface).
\end{proof}

\begin{rem}
   If one removes the word "prime" from each of the above theorems, then conclusions (ii) and (iii) become less refined and we may only conclude that $\partial N_D$ (or equivalently, the ideal contact boundary $\partial_\infty X$) contains an embedded Klein bottle.
\end{rem}

In the unobstructed case, the diffeomorphism type of the contact boundary is sufficient to characterize all concave divisor neighborhoods via direct application of \cite{Neumann:calculus}.

\begin{proof}[\normalfont \textbf{Theorem~\ref{thm:unique_nbhd}}]
Let $N_D$ and $N_{\widetilde{D}}$ denote concave divisor neighborhoods associated to a pair of unobstructed divisors $D$ and $\widetilde{D}$. Then if their divisor boundaries $\partial N_D$ and $\partial N_{\widetilde{D}}$ are diffeomorphic, the manifolds $N_D$ and $N_{\widetilde{D}}$ are diffeomorphic up to blow-ups.
\end{proof}

The combined results about obstructed prime divisors and unobstructed divisors give a complete topological classification of neighborhoods of concave $SNC^+$ divisors. 

With additional consideration, we may apply this to produce our most general result for divisor compactifications. As discussed, the divisor compactification $(M,D)$ admits a decomposition of the form 
\[M = \overline{X} \cup_{\Phi} N_D\]
where $(\overline{X}, d\lambda)$ is a Liouville domain, $(N_D,\omega)$ is a concave neighborhood of $D$ and $\Phi \colon (\partial \overline{X}, \xi) \to (\partial N_D,\xi)$ is a contactomorphism defining the gluing. By considering divisor boundaries and ideal contact boundaries, we may consider this to be a contactomorphism $\Phi \colon (\partial_\infty X, \xi) \to (Y_D,\xi_D)$. Thus we may consider $\Phi$ to be an element of Cont($\partial_{\infty} X, \xi$), the contactomorphism group of the ideal contact boundary. The isotopy class $[\Phi] \in \pi_0($Diff$(\partial_\infty X)$ determines the diffeomorphism type of $M$. This plus \Cref{thm:unique_nbhd} is enough to conclude:
\begin{proof}[\normalfont \textbf{Theorem~\ref{thm:main}}]
\renewcommand{\qedsymbol}{}
\em{
    Let $D$ be a concave compactifying $SNC^+$ divisor for a 4-dimensional Liouville domain $(X^4,\omega)$ and let $(M^4,\omega)$ be a compactification of $X$ by an $SNC^+$ divisor $D$ with associated mapping class $[\Psi] \in \pi_0(Cont(\partial X))$ defining the capping. Then $M$ either satisfies at least one of:
    \begin{enumerate}[label=(\roman*)]
        \item $(M,\omega)$ is a blow-up of a ruled symplectic manifold,
        \item $\partial X$ is Seifert-fibered over the $\RP^2$ via a fibration with at most one singular fiber,
        \item $\partial X$ contains a virtually essential Klein bottle
    \end{enumerate}
    or any other $SNC^+$ divisor compactification $(\widetilde{M}, \widetilde{\omega})$ with the same mapping class $[\Psi]$ can be obtained from $(M,\omega)$ via blow-ups and blow-downs.
    }
\end{proof}
This result could be strengthened if one could understand to what extent the contact mapping class $[\Psi]$ may vary. For example, if every diffeomorphism $\Psi \colon \partial N_D \to \partial N_D$ extends to a diffeomorphism $\widetilde{\Psi} \colon N_{D} \to N_{D}$, then the diffeomorphism type of any divisor compactification of $X$ is completely determined by the diffeomorphism type of $X$ and the topology of $D$, independent of the choice of contactomorphism defining the gluing. We will use this to finish a proof of \Cref{thm:main_example} in the next section.

\subsection{0-spheres and Giroux torsion}

We will only consider connected plumbing graphs $\Gamma$ for the remainder of the discussion. For now, we also assume that $\Gamma$ is without loops i.e. without configurations of the form
 
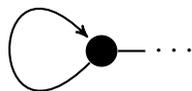
\[\begin{tikzpicture}[scale=0.5]
    \GraphInit[vstyle=Empty]
    \SetGraphUnit{2}
    \Vertex[NoLabel]{A}
    \EA[L=$\cdots$](A){B}
    \AddVertexColor{black}{A}

    \Edges(A,B)
    \Loop(A)

  \end{tikzpicture}\]

Let $Y_\Gamma$ be the 3-manifold plumbing obtained from $\Gamma$. When $\Gamma$ is the graph associated to a symplectic normal crossing divisor $D \subset M$, the positive (resp. negative) GS-criterion says that $D$ has a concave (convex) regular neighborhood if and only if the linear equation
\[Q_D\uline{b} = \uline{a}\]
has a solution $\uline{b} \in \RR^{N}$ with all positive (negative) entries. In both these cases, the concave/convex neighborhood furnishes the boundary plumbing with a natural contact structure which is co-oriented relative to the boundary orientation. 

The anatomy of this contact structure is outlined in but briefly, the contact boundaries are built from pieces of two types:
\begin{itemize}
    \item Type I Pieces: which are diffeomorphic to $S^1$-bundles over surfaces with boundary.
    \item Type II Pieces: neighborhoods of the plumbing tori which are diffeomorphic to $T^2 \times I$ equipped with a minimally twisting contact structure.
\end{itemize}

The induced contact structures Type I pieces are regular in the sense of Boothby and Wang and so their Reeb foliations provide them with a natural $S^1$-fibration structure. Type II pieces are simple interpolations between the characteristic slopes of the torus boundaries of two neighboring Type I pieces, connected according to $\Gamma_D$. Without much additional work, one can construct similar contact structures on arbitrary plumbings.

\begin{thm}
Let $\Gamma$ be a plumbing graph. Then there exists a contact structure $\xi_{\Gamma}$ on the 3-manifold plumbing $Y_\Gamma$ which is transverse to the $S^1$-fibration away from the plumbing tori and minimally twisting near the plumbing tori. 
Any other contact structure on $Y_\Gamma$ satisfying these conditions is contactomorphic to $\xi_\Gamma$. In the case when $\Gamma$ satisfies the positive/negative GS-criterion, $\xi_\Gamma$ is co-orientable and is contactomorphic to the natural contact structure induced by the concave/convex neighborhood.
\end{thm}
\begin{proof}
We proceed to build our contact manifolds out of two types of pieces as before: 
\begin{itemize}
\item Type I pieces are diffeomorphic to $S^1$-bundles over surfaces with boundary with horizontal contact structures.

\item Type II Pieces are diffeomorphic to $T^2 \times I$ with minimally twisting contact structures.
\end{itemize}

We will have one Type I piece $Y_v$ for each vertex $v$ in $\Gamma$. The contact manifold $Y_v$ will be obtained by removing solid tori from a closed contact manifold. That closed contact manifold will depend on the decorations $(g,k)$ of the vertex $v$. We will take different approaches depending on whether $k = 0$ or $k \neq 0$. We will handle the case $k \neq 0$ first. 

When $D$ is a smooth divisor, the only entry in the intersection form $Q_D$ is $k := D\cdot D$ and so the GS-criteria become equivalent to the requirement that the self-intersection number take a definite sign. In the case when $k > 0$, $D$ admits a concave neighborhood. In the case $k < 0$, it admits a convex one. 

In both cases, the induced contact manifold $(\partial N_D, \xi_D)$ admits the structure of a Boothby-Wang bundle over the divisor, i.e. a bundle map $\pi \colon (\partial N_D, \xi_D) \to (D,\omega)$ which satisfies $\pi^*d\alpha = \omega$. In this setting, the Reeb vector field is everywhere tangent to the $S^1$-fibers and the bundle map $\pi$ is obtained as a quotient of the periodic flow of the Reeb vector field. Because of the co-orientation of the contact structure, in both the $k > 0$ and $k < 0$ cases, the bundle is $\pi \colon (\partial N_D, \xi_D) \to (D,\omega)$ is the principal $S^1$-bundle with Euler number $e(\partial N_D, \xi_D) = |k|$. 

We start with the bundle $\pi \colon (\partial N_D, \xi_D) \to (D,\omega)$ and remove small solid tori neighborhoods of a collection of $d_v$ disjoint fibers where $d_v$ is the degree of the vertex. This leaves us with a contact $S^1$-bundle over a surface with boundary with Reeb fibers. This piece has $d_v$ boundary components, one for each solid torus that was removed.

Finally, there is the case when $k = 0$. If $\Sigma$ is a closed surface, by \cite{ET:foliations} and \cite{Honda:tight2} there is a unique contact structure on $\Sigma \times S^1$ which is $C^0$-close to the folation $\Sigma \times \{pt\}$. In the case when $g(\Sigma) > 0$, this contact structure is horizontal with respect to the natural $S^1$-fibration structure. If $g(\Sigma) = 0$, this is the unique almost-horizontal contact structure on $S^2 \times S^1$. In either case, we start with $\Sigma \times S^1$ and remove small solid torus neighborhoods of a collection of $d_v$ disjoint fibers for which the contact structure is horizontal nearby.

We construct pieces of Type II by interpolating between the characteristic slopes of the boundary tori of two neighborhood Type I pieces, connected accordint to $\Gamma$.

It is clear from construction that our contact structures agree with those induced by concave/convex neighborhoods of divisors. The fact that this contact structure is globally co-orientable follows from the co-orientability of the induced structures.
    
\end{proof}

We call any realization $(Y_{\Gamma}, \xi_\Gamma)$ of this unique contactomorphism type a \emph{contact plumbing}. Since the natural contact structures on concave divisors satisfy these hypotheses, we have:
\begin{cor}
Any convex/concave plumbing boundary $(\partial N_D, \xi)$ is contactomorphic to the contact plumbing $(Y_D,\xi_D)$ determined by the divsor graph $\Gamma_D$.
\end{cor}
We conclude that the contact geometry of concave divisor boundaries is topologically determined by the divisor graph $\Gamma_D$. We are thus free to draw conclusions about the contactomorphism type of a concave divisor boundary from topological information about the graph $\Gamma_D$. One such conclusion which one can reach is:

\begin{thm}
    If $(Y_{\Gamma},\xi_\Gamma)$ is a contact plumbing and $\Gamma$ contains a configuration of the form
 
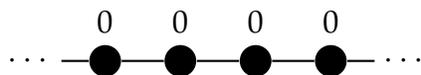
\[\begin{tikzpicture}[scale=0.5]
    \GraphInit[vstyle=Empty]
    \SetGraphUnit{2}
    \Vertex[L=$\cdots$]{A}
    \EA[LabelOut, Lpos=90, L=$0$](A){B}
    \EA[LabelOut, Lpos=90, L=$0$](B){C}
    \EA[LabelOut, Lpos=90, L=$0$](C){D}
    \EA[LabelOut, Lpos=90, L=$0$](D){E}
    \EA[L=$\cdots$](E){F}
    \AddVertexColor{black}{B,C,D, E}
    
    \Edges(A,B,C,D,E,F)

  \end{tikzpicture}\]
    which is not part of a component of $\Gamma$ of the form
 
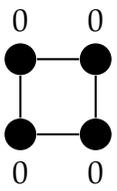
\[\begin{tikzpicture}[scale=0.5]
    \GraphInit[vstyle=Empty]
    \SetGraphUnit{2}
    \Vertex[LabelOut, Lpos=90, L=$0$]{A}
    \EA[LabelOut, Lpos=90, L=$0$](A){B}
    \SO[LabelOut, Lpos=-90, L=$0$](A){C}
    \SO[LabelOut, Lpos=-90, L=$0$](B){D}
    \AddVertexColor{black}{A,B,C,D}
    
    \Edges(A,B,D,C,A)

  \end{tikzpicture}.\]
    then $(Y_{\Gamma}, \xi_{\Gamma})$ has nontrivial Giroux torsion.
\end{thm}
\begin{proof}
Any portion of a 3-manifold boundary corresponding to the configuration 
 
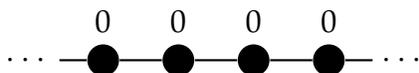
\[\begin{tikzpicture}[scale=0.5]
    \GraphInit[vstyle=Empty]
    \SetGraphUnit{2}
    \Vertex[L=$\cdots$]{A}
    \EA[LabelOut, Lpos=90, L=$0$](A){B}
    \EA[LabelOut, Lpos=90, L=$0$](B){C}
    \EA[LabelOut, Lpos=90, L=$0$](C){D}
    \EA[LabelOut, Lpos=90, L=$0$](D){E}
    \EA[L=$\cdots$](E){F}
    \AddVertexColor{black}{B,C,D, E}
    
    \Edges(A,B,C,D,E,F)

  \end{tikzpicture}\]
admits a local toric contact structure. Each 0-framed sphere contributes $> \frac{\pi}{2}$-torsion to the contact structure and so we will have $> 2\pi$ torsion except in the case when $\Gamma_D$ is of the form 
 
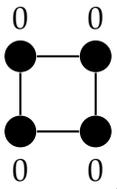
\[\begin{tikzpicture}[scale=0.5]
    \GraphInit[vstyle=Empty]
    \SetGraphUnit{2}
    \Vertex[LabelOut, Lpos=90, L=$0$]{A}
    \EA[LabelOut, Lpos=90, L=$0$](A){B}
    \SO[LabelOut, Lpos=-90, L=$0$](A){C}
    \SO[LabelOut, Lpos=-90, L=$0$](B){D}
    \AddVertexColor{black}{A,B,C,D}
    
    \Edges(A,B,D,C,A)

  \end{tikzpicture}.\]
for which we have $Y_D \approx T^3$ and $\xi_D$ contactomorphic to the universally tight contact structure on $T^3$.
\end{proof}

If a plumbing graph $\Gamma$ admits such a configuration, we say that it has \emph{positive torsion}. Any divisor $D$ whose divisor graph $\Gamma_D$ has positive torsion is also said to have positive torsion. This result has symplectic-topological implications. If $(M,\omega)$ is a closed symplectic 4-manifold and $D \subset M$ is a concave divisor, then any concave divisor neighborhood $N_D \subset M$ determines a weak filling $X = M - N_D$ of $(\partial N_D, \xi_D)$. Since Giroux torsion is an obstruction to weak fillability \cite{Gay:handles}, we immediately conclude: 

\begin{cor}
    If $\Gamma$ is a plumbing graph with positive torsion, then $\Gamma$ cannot appear as the divisor graph of an \underline{embedded} concave divisor $D \subset M$. In other words, concave divisors with positive torsion cannot be embedded into closed symplectic 4-manifolds.
\end{cor}

By contrast, every convex symplectic divisor embeds into a closed symplectic 4-manifold since we may always cap off a convex divisor neighborhood.

\section{Application: A non-algebraic Stein manifold}\label{sec:example}

Let $M_{KT}$ denote the Kodaira-Thurston manifold (see \cite{thurston:example}) defined as $M_{KT} := \MMM_{\phi} \times S^1$ 
where $\MMM_{\phi}$ is the mapping torus of a right Dehn twist $\phi \colon T^2 \to T^2$. We give $\MMM_{\phi}$ local coordinates $(\theta_1, \theta_2, \theta_3)$ and so $M_{KT}$ may be given local coordinates $(\theta_1, \theta_2, \theta_3, \theta_4)$ and we define a symplectic form
\[\omega := d\theta_1 \wedge d\theta_2 + d\theta_3 \wedge d\theta_4.\]

We note that $(M_{KT},\omega)$ contains a natural $SNC^+$ divisor. Let $\sigma \colon S^1 \to \MMM_{\phi}$ be a section of the mapping torus and let $\pi \colon \MMM_{\phi} \to S^1$ be the natural projection map. We define two tori 
\[D_L := \pi^{-1}(pt) \times \{pt\}\]
\[D_R := \sigma(S^1) \times S^1\]
and let $D := D_L \cup D_R$. Evidently, $D_L$ and $D_R$ are $\omega$-orthogonal symplectic divisors and so $D$ is an $SNC^+$ divisor. 

The plumbing graph associated to the divisor $D$ is 
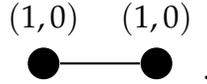
\[
  \begin{tikzpicture}[scale=0.5]
    \GraphInit[vstyle=Empty]
    \SetGraphUnit{3}
    \Vertex[LabelOut, Lpos=90,L={$(1, 0)$}]{A}
    \EA[LabelOut, Lpos=90, L={$(1, 0)$}](A){B}
    
    \Edges(A,B)
    \AddVertexColor{black}{A,B}
  \end{tikzpicture}.
\]
and so the intersection form is given by 
\[
    Q_D = 
        \begin{bmatrix}
            0 & 1\\
            1 & 0
        \end{bmatrix}
\]
which admits many GS-positive solutions (regardless of the symplectic area chosen for $D_L$ and $D_R$). It follows that $N_D$ is concave and so it has contact-type boundary.

We let $X := M_{KT} - D$. Since $X$ is a symplectic punctured torus bundle over a symplectic punctured torus, $X$ is a Liouville manifold (in fact, a Stein domain) \cite[Theorem~1.2]{BVHM:liouville}. The ideal contact boundary $\partial_{\infty} X$ is contactomorphic to $\partial N_D$. 
\begin{proof}[\normalfont \textbf{Theorem~\ref{thm:main_example}}]
\renewcommand{\qedsymbol}{}
\em{
    The Liouville manifold $X$ is not biholomorphic to any affine variety.
    }
\end{proof}
Whose proof will rely on the main application of our result.
\begin{prop}\label{prop:unobstructed}
Any concave $SNC^+$ compactifying divisor $\widetilde{D}$ for $X_{KT}$ is unobstructed and prime.
\end{prop}
Recall that a divisor $\widetilde{D}$ is unobstructed if $\Gamma_{\widetilde{D}}^{\xi} \approx \Gamma_{\widetilde{D}}^{Top}$ and ``prime" simply means that each divisor boundary is a prime 3-manifold. The fact that all compactifying divisors for $X$ are unobstructed allows us to apply \Cref{thm:unique_nbhd}. From there, we verify:
\begin{prop}\label{prop:extends}
If $\widetilde{D}$ is a compactifying divisor for $X_{KT}$, then Every diffeomorphism $\Psi \colon \partial N_{\widetilde{D}} \to \partial N_{\widetilde{D}}$ extends to a diffeomorphism $\widetilde{\Psi} \colon N_{\widetilde{D}} \to N_{\widetilde{D}}$.
\end{prop}
By our previous discussion, may conclude: 
\begin{thm}\label{thm:unique}
    Let $(\widetilde{M},\widetilde{\omega})$ be a concave divisor compactification of $X_{KT}$. Then, up to blow-ups, $M$ is diffeomorphic to $M_{KT}$. 
\end{thm}
From these results, we may prove our main theorem:
\begin{proof}[\textbf{Proof of Theorem ~\ref{thm:main_example}}]
    Suppose, on the contrary, that $X$ is affine. Then it follows that it admits an algebraic $SNC^+$ compactification $(\widetilde{M},\widetilde{D})$. By \Cref{prop:unobstructed} we see that the divisors $D$ and $\widetilde{D}$ are  both unobstructed. We may thus apply \Cref{thm:unique_nbhd} and see that we may perform a sequence of blow-ups on $\widetilde{M}$ until our divisor $\widetilde{D}$ becomes diffeomorphic to $D$. This process results in a 4-manifold $\widehat{M}$. It then follows from \Cref{prop:extends} that $\widehat{M}$ is diffeomorphic to $M_{KT}$. This cannot possibly occur since $b_1(M_{KT}) = 3$ and blowing up does not alter the first Betti number.
\end{proof}
We will need a few different lemmas in order to establish \Cref{thm:unique}:
\begin{enumerate}[label=(\roman*)]
\item The divisor boundary $Y_D$ is a prime 3-manifold
\item $Y_D$ does not contain any embedded Klein bottles
\item $X$ does not admit a compactification $(M,D)$ such that $D$ contains a sphere with self intersection number 0. 
\end{enumerate}
Together, items (i) and (ii) tell us that no divisor with divisor boundary diffeomorphic to $Y_D$ can contain Klein bottle pieces. Item (iii) allow us to avoid issues involving $0$-curves as they cannot be a component in any divisor compactification of $X$. In particular, this implies that no 0-curves appear in the contact chain reduction of any concave divisor graph with divisor boundary diffeomorphic to $Y_D$. From all of this, it will follow that the contact chain reduction of any such divisor graph must be isomorphic to its topological reduction which is the key to concluding our result. We will prove the lemmas listed above in sequence.

\subsection{Prime divisor boundary}
\begin{lem}\label{lem:prime}
The manifold $Y_D$ is prime.
\end{lem}
\begin{proof}
Suppose not and suppose we may compose $Y_D$ as a non-trivial connected sum. Let $S \subset Y_D$ denote the embedded $S^2$ defining the connected sum decomposition. Recall that $Y_D$ is the result of a gluing of two $S^1$-bundles over punctured tori. We will let $Y_L$ and $Y_R$ denote these bundles. Let $T \subset Y_D$ be the plumbing torus, i.e. the image of the boundary tori of $Y_L$ and $Y_R$ under this gluing.

We wish to understand how $S$ intersects $T$. After perturbing via an isotopy, we may assume that $S$ and $T$ intersect transversely. By compactness, we know that $S$ and $T$ intersect along a collection of closed curves. Removing $T$ splits $S$ into a collection of surfaces each having some number of boundary components, each of which is a curve in the collection. Since $S$ is a sphere, one of these pieces must be a disk. 

The plumbing torus $T$ is an incompressible surface and so the boundary of this disk must also bound an embedded disk in $T$. This implies that each of the curves in our collection are contractible in $T$. Since $Y_D$ is prime, these two disks bound a 3-ball and so we may use this 3-ball to push the disk across $T$ and isotope it until it lies on only one side of the decomposition. We may repeat this argument for the remaining collection and eliminate every intersection, pushing everything to the same side of $T$. It follows that we may assume that $S$ lies on one side of the plumbing torus or the other. This would allow us to decompose one of the sides as a non-trivial connected sum. Since both $\pi_2(Y_L) \approx \pi_2(Y_R) \approx 0$, we know this is impossible.
\end{proof}

\subsection{No Klein bottles}
We will now prove that $Y_D$ is free of Klein bottles.
\begin{lem}\label{lem:KT_1}
    The manifold $Y_D$ contains no embedded Klein bottles. 
\end{lem}
\begin{proof}
    We start by showing that $Y_D$ cannot contain compressible Klein bottles. By \Cref{lem:prime} and \Cref{prop:compressible_klein1}, if $Y_D$ did contain a virtually compressible Klein bottle, then it may be given the structure of a Seifert fibration with at most one singular fiber. If this were the case, its fundamental group would admit the following presentation:
\[\pi_1(Y_D) = \bracket{\mu, \lambda : \mu\lambda\mu^{-1} = \lambda^{-1}, \mu^{2a} = \lambda^{-b}}\]
for some $a,b \in \ZZ$.

    We understand the topology of $Y_D$ as the result of plumbing surgery on a pair of trivial $S^1$-bundles over $T^2$. From this description, we may give a presentation of the fundamental group (c.f. \cite[Chapter~5.3]{Orlik:seifert}):
    \[\pi_1(Y_D) = \bracket{u,v,x,y, F_L, F_R : F_L = [x,y], F_R = [u,v]}.\]
    By abelianizing, we see that
    \[H_1(Y_D) = \pi_1(Y_D)_{Ab} = \ZZ^4\] 
    so the abelianization of the second presentation is obviously not equal to the abelianization of the first. Thus we arrive at a contradiction.

    We then turn to the case of embedded incompressible Klein bottles. Topologically, $Y_D$ is the boundary of a plumbing of two trivial $\DD^2$-bundles over $T^2$. We let $T \subset Y_D$ be a fixed representative of the sole plumbing torus. Suppose that $K \subset Y_D$ is an embedded virtually incompressible Klein bottle.

    Cutting $Y_D$ along $T$ splits the manifold into two pieces $Y_L$ and $Y_R$. Both of these manifolds are diffeomorphic to $S^1$-bundles over a torus minus an open disk. This cut splits $K$ into a collection $\{\Sigma_{i}\}_{i \in \mathcal{I}}$ of surfaces with boundary. A fixed boundary component of any one of these surfaces is some closed curve in $T$. Since $K$ is an essential surface, each $\Sigma_{i}$ is an incompressible and boundary incompressible surface in either $Y_L$ or $Y_R$. These surfaces must be either horizontal or vertical with respect to the Seifert fibrations (c.f. \cite{Hatcher:3manifolds}, Proposition 1.12). Recall a surface $\Sigma$ in a Seifert-fibered space $Y$ with Seifert-fibration $\pi:Y\to S$ is \emph{vertical} if $\Sigma$ is a union of fibers of $\pi$ and \emph{horizontal} if it is transverse to the fibers of $\pi$. If $\Sigma$ is horizontal, the restriction of $\pi$ gives $\Sigma$ the structure of a branched cover over $S$.

Since the plumbing torus $T$ is incompressible, we can assume that none of the $\Sigma_i$ are disks. Thus since $K$ is broken into the pieces $\Sigma_i$ by $T$, we know that each $\Sigma_i$ must be either: 
\begin{itemize}
        \item a mobius band,
        \item a punctured Klein bottle,
        \item an annulus, or
        \item a Klein bottle.
    \end{itemize}
If $\Sigma_i$ is a horizontal surface, then it must be a branched cover over a punctured $T^2$ and hence it must be oriented and have negative Euler characteristic (by the Riemann-Hurwitz formula, \cite[Theorem~2.5.2]{Jost:riemann}). Thus none of the $\Sigma_i$ are horizontal.

A vertical surface must be either an annulus or a torus so $K$ is a union of vertical annuli in $Y_L$ and $Y_R$. In order for these pieces to glue to form a Klein bottle, we must always have a pair of pieces $\Sigma_L \subset Y_L$ and $\Sigma_R \subset Y_R$ which are glued together in the union. Since the gluing between $Y_L$ and $Y_R$ sends the fiber of one side to a section on the other side, we cannot have that both of these surfaces are vertical, a contradiction. 
\end{proof}

\subsection{No embeddings into ruled surfaces}
In order to establish (iii) we need the following.
\begin{lem}\label{lem:positive}
The cohomology group $H_2(X)$ contains a class with positive self intersection.
\end{lem}
\begin{proof}
We consider our chosen section $\sigma \colon S^1 \to \MMM_{\phi}$ of the mapping torus and we let $\pi \colon \MMM_\phi \to S^1$ be the natural projection. As above, we have
\[D_L := \pi^{-1}(pt) \times \{pt\}\]
and
\[D_R := \sigma(S^1) \times S^1.\]
We may trivialize the $T^2$-bundle $\pi \colon M_{KT} \to T^2$ near the fiber $D_L$ in order to obtain generators for nearby $T^2$-fibers. Given $q \in S^1$ near our chosen point $p := \pi(D_L)$, we let $D_q$ denote the fiber lying over it. We let $\mu_q,\lambda_q \in H_1(D_q)$ denote the generators obtained from the trivialization. We choose our trivialization such that $\mu_q$ is parallel to the curve along which we perform a Dehn twist in each mapping torus $\MMM_{\phi} \times \{pt\} \subset M_{KT}$.

We let $A_L \in H_2(M_{KT})$ denote the homology class represented by a product of $\lambda_q$ with the circle $\{pt\} \times S^1 \subset M_{KT}$. We let $A_R \in H_2(M_{KT})$ denote the homology class represented by a product of a parallel copy of $\sigma(S^1) \times \{pt\}$ (i.e. one disjoint from $D_R$) with $\mu$, the image of $\mu_{q} \times I$ in $M_{\phi} \subset M_{KT}$.  We let $A = A_L + A_R$. Evidently, we have $A\cdot A > 0$ which completes the proof.
\end{proof}
With this, we may prove:
\begin{lem}\label{lem:KT_3}
    $X$ cannot admit a divisor compactification $(M,D)$ such that $D$ contains a sphere of self intersection number zero.
\end{lem}
\begin{proof}
Assume by way of contradiction that $X$ admits such a compactification $(M,D)$. Let $S \subset D$ denote a sphere with $S\cdot S = 0$. By \cite[Corollary~1.3(ii)]{mcduff:rational}, the manifold $M$ is diffeomorphic to an $S^2$-bundle over a closed surface $\Sigma$ blown up some number of times and $S$ is homologous to a fiber of this bundle. It is straightforward to compute the intersection form of this manifold:

Let $\widetilde{M}$ be the minimal symplectic manifold obtained by blowing down all of the exceptional curves of $M$ (so that $\widetilde{M}$ is a genuine $S^2$-bundle over a surface). The second homology group $H_2(\widetilde{M};\ZZ)$ is generated by a generic fiber $F$ and a generic section $\widetilde{\Sigma} \subset \widetilde{M}$. If $\widetilde{M}$ is trivial, then the intersection form is given in the $\{[F], [\widetilde{\Sigma}]\}$-basis by
\[
    \begin{bmatrix}
        0 & 1\\
        1 & 0\\
    \end{bmatrix}
\]
and if $\widetilde{M}$ is non-trivial, the intersection form is given by
\[
    \begin{bmatrix}
        0 & 1\\ 
        1 & 1\\
    \end{bmatrix}.
\]
In order to obtain $M$ from $\widetilde{M}$, we must perform blow-ups at a collection of disjoint points $\{p_1,\ldots, p_m\} \subset \widetilde{M}$. We may assume these points are disjoint from $F$ and $\widetilde{\Sigma}$. After performing these blow-ups, the intersection form of $M$ is given by either
\[
    \begin{bmatrix}
        0 & 1\\
        1 & 0\\ 
    \end{bmatrix} \oplus -\mathbb{1}[m]
\]
or
\[
    \begin{bmatrix}
        0 & 1\\
        1 & 1\\
    \end{bmatrix} \oplus -\mathbb{1}[m]
\]
where $-\mathbb{1}[m]$ denotes an $m\times m$ matrix with $-1$'s along the diagonal and zeros off the diagonal.

We consider $Z = \widetilde{M} - F$. Since 
$F$ is parallel to the 0-framed sphere in $D$, 
we must have $X \subset Z$. We note that 
the intersection product
\[- \cdot - \colon H_2(Z) \otimes H_2(Z) \to \ZZ\]
is negative semi-definite. To see this, we simply observe that $Z$ is diffeomorphic (via our generic section $\widetilde{\Sigma}$) to a blow-up of $\widetilde{\Sigma}^* \times S^2$ where $\widetilde{\Sigma}^* = \widetilde{\Sigma} - \{pt\}$. Every class $\beta \in H_2(Z)$ has 
$\beta \cdot \beta \le 0$ so the induced intersection product $H_2(Z) \otimes H_2(Z) \to \ZZ$ is negative semi-definite. The pushforward of the class from \Cref{lem:positive} along the inclusion is a class with positive self intersection and so we have a contradiction.

\end{proof}

\subsection{Mapping class considerations and conclusion}
In order to apply the full strength of our main theorem as in the conclusion of the previous section, we prove:
\begin{proof}[\normalfont \textbf{Proposition~\ref{prop:extends}}]
\renewcommand{\qedsymbol}{}
\em
Every diffeomorphism of $Y_D$ extends to the neighborhood $N_D$ i.e. for any diffeomorphism $\psi \colon Y_D \to Y_D$, there exists a diffeomorphism of $\Psi \colon N_D \to N_D$ such that $\Psi|_{Y_D} = \psi$.
\end{proof}
\begin{proof}
    We note that $Y_D$ is not a Seifert fibered space. This follows, for example, from a corollary of Neumann's theorem (\cite{Neumann:calculus}, Corollary 5.7) which states that the normal form of a plumbing graph associated to a Seifert fibered space must be star shaped (i.e. a single genus $g$ component connected to a number of chains of spheres). $\Gamma_D$ is in normal form but it is certainly not star shaped.

    If $T \subset Y_D$ is a plumbing torus representative, then $Y_D - T$ is a pair of $S^1$-bundles over surfaces with boundary as above. It follows that each piece in this decomposition is Seifert-fibered. 

    It follows that the singleton set $\{T\}$ yields a JSJ-decomposition of $Y_D$ (c.f. \cite{Hatcher:3manifolds}, Chapter 2). By the uniqueness theorem for JSJ-decompositions, we may assume that $\psi$ is isotopic to a diffeomorphism which fixes a regular neighborhood of $T$. Thus $\psi$ is isotopic to a diffeomorphism of the two pieces $Y_L$ and $Y_R$ which is the identity near the boundary. This diffeomorphism respects the Seifert fibrations of the two pieces and so is induced by a pair of diffeomorphisms $f_L,f_R \colon T^2_* \to T^2_*$ of the $T^2_*$, the torus minus a disk. We may use these diffeomorphisms to define our diffeomorphism $\Psi \colon N_D \to N_D$ which completes the proof.
\end{proof}
Once we know that every diffeomorphism extends, we may reach our strongest possible conclusion about the collection of all compactifications.
\begin{thm}\label{thm:main_application}
    Let $(M,\widetilde{\omega})$ be a concave divisor compactification of $X_{KT}$. Then, up to blow-ups, $M$ is diffeomorphic to $M_{KT}$. 
\end{thm}
\begin{proof}
    Given such a compactification $(M,\widetilde{\omega})$ by a divisor $\widetilde{D}$, we let $\Gamma_{\widetilde{D}}$ denote its associated divisor graph. Since $M_{KT}$ and $M$ are both divisor compactifications of the same manifold $X_{KT}$, it follows that $Y_D$ and $Y_{\widetilde{D}}$ are diffeomorphic. Since $X_{KT}$ is a Liouville manifold, it follows that $(Y_{D}, \xi_D)$ and $(Y_{\widetilde{D}}, \xi_{\widetilde{D}})$ are also contactomorphic to $\partial_{\infty} X_{KT}$, and thus to each other. 

    Since $Y_D$ is unobstructed, the contact reductions $\Gamma_D^{\xi}$ and $\Gamma_{\widetilde{D}}^{\xi}$ are both isomorphic to their respective topological reductions. By Neumann's theorem (and the modified chain reduction lemma), it follows that we may relate the divisor graph of $\Gamma_{D}$ to $\Gamma_{\widetilde{D}}$ via a sequence of blow-ups and blow-downs. From this, we may perform corresponding blow-ups on $(M,\widetilde{\omega})$ in order to transform the divisor $\widetilde{D}$ into $D$ (this follows because $\Gamma_D$ is already in topological normal form).
    This presents a blow-up $B\ell(M)$ as a compactification of $X_{KT}$ by the same divisor $D$ and so $M$ and $B\ell(M)$ differ only by the isotopy class of diffeomorphism of $Y_D$ defining these compactifications. By \Cref{prop:extends}, there is only one such isotopy class. It follows that $B\ell(M)$ and $M_{KT}$ are diffeomorphic
\end{proof}
Thus we may conclude:
\begin{proof}[\normalfont \textbf{Theorem~\ref{thm:main_example}}]
\renewcommand{\qedsymbol}{}
{\em The manifold $X = M_{KT} - D$ is not biholomorphic to any affine variety.}
\end{proof}

A similar argument should work for any Liouville manifold whose ideal contact boundary is prime and whose divisors are all unobstructed, so long as it admits a compactification $M$ 
with $b_1(M)$ odd. In light of \cite{McLean:affine} and \cite{Seidel:symplectic}, this is the first example of a non-affine 
symplectic manifold whose obstruction from being affine cannot be detected by the growth rate 
of symplectic homology. The existence of this example tells us that there should be a deeper obstruction to being affine that assumes the same basic topological setup as above and is not captured via growth rate techniques. The existence and nature of such an obstruction is still currently unknown.

\bibliography{arxiv}

\end{document}